\documentclass[11pt] {article}
\usepackage{amsmath,amssymb,amsfonts}
\usepackage[utf8]{inputenc} 
\usepackage{amsthm}

\usepackage{tikz}
\usepackage{hyperref}
\usepackage{graphicx}
\usepackage{amsfonts}

\allowdisplaybreaks

\usepackage{color}
\usepackage[english]{babel}
\usepackage{ifthen}
\usepackage{graphicx}
\usepackage{tikz}
\usetikzlibrary{decorations.markings}
\usetikzlibrary{plotmarks}
\usetikzlibrary{patterns}
\usepackage{pgfplots}
\usepgfplotslibrary{fillbetween}
\usepackage{makeidx}
\usetikzlibrary{calc}

\usepackage{geometry}
\usepackage{graphicx} 

\usepackage{array} 
\usepackage{paralist} 
\usepackage{verbatim} 
\usepackage{subfig} 

\usepackage{textpos}

\usepackage[style=numeric,
            sorting=nyt,
            backref=true,
            backrefstyle=three,
            isbn=false,
            backend=biber]{biblatex}
\usepackage[style=numeric]{biblatex}

\newcommand{\R}{\mathbb{R}}

\tikzset { domaine/.style 2 args={domain=#1:#2} }

\newtheorem{theorem}{Theorem}

\newtheorem*{corollary}{Corollary}
\newtheorem{theo}{Theorem}[section]
\newtheorem*{theo*}{Theorem}

\newtheorem{prop}[theo]{Proposition}
\newtheorem*{prop*}{Proposition}
\newtheorem{lem}{Lemma}[section]
\newtheorem{cor}{Corollary}[section]
\newtheorem*{cor*}{Corollary}
\newtheorem{de}{Definition }

\newtheorem{claime}{Claim}
\newtheorem{remark}{Remark}[section]

\newcommand{\nocontentsline}[3]{}
\newcommand{\tocless}[2]{\bgroup\let\addcontentsline=\nocontentsline#1{#2}\egroup}

\title{Rigidity Theorems for Asymptotically Euclidean $Q$-singular Spaces}
\author{R. Avalos \footnote{Department of Mathematics, Universität Tübingen}, P. Laurain\footnote{Université Gustave Eiffel, Laboratoire d'Analyse et de Mathématiques Appliquées}, N. Marque\footnote{Université de Lorraine, Institut Elie Cartan de Lorraine}}

\begin{document}
\maketitle

\begin{abstract}
In this paper we prove some rigidity theorems associated to $Q$-curvature analysis on asymptotically Euclidean (AE) manifolds, which are inspired by the analysis of conservation principles within fourth order gravitational theories. A central object in this analysis is a notion of fourth order energy, previously analysed by the authors, which is subject to a positive energy theorem. We show that this energy can be more geometrically rewritten in terms of a fourth order analogue to the Ricci tensor, which we denote by $J_g$. This allows us to prove that Yamabe positive $J$-flat AE manifolds must be isometric to Euclidean space. As a byproduct, we prove that this $J$-tensor provides a geometric control for the optimal decay rates at infinity. This last result reinforces the analogy of $J$ as a fourth order analogue to the Ricci tensor.  
\end{abstract}

\section{Introduction}

In this paper we intend to analyse rigidity properties associated to fourth order geometric operators defined on a Riemannian manifold $(M^n,g)$, $n\ge 3$. In particular, we are interested in asymptotically Euclidean (AE) manifolds. These are complete non-compact Riemannian manifolds $(M^n,g)$, where $M$ consists of a compact core $K$ and, outside of it, there is a diffeomorphism $\Phi$ onto the exterior of a ball in $\mathbb{R}^n$. That is, $M\backslash K\cong \mathbb{R}^n\backslash \overline{B_1(0)}$. In this setting, the metric $g$ is supposed to approach (as we move towards infinity) the Euclidean metric pulled back from $\mathbb{R}^n$ to $M^n$ via $\Phi$.\footnote{For further details, see Definition \ref{AEmanifolds}.} These geometric structures are extremely well motivated from General Relativity (GR), where they model isolated gravitational systems, and have proven to be at the center of several interesting problems in Riemannian geometry \cite{Schoen1,MR849427,Eichmair1,Eichmair2,Carlotto1,Carlotto2,Carlotto3,Carlotto4-2,Carlotto4,Lee-Parker,Yau1,Corvino1,MR2225517}. Most notably, they play a central role in the resolution of the Yamabe problem due to the connection with the Positive Mass Theorem (PMT) in GR (\cite{SY1,SY2,SY3,Lee-Parker}).

It is a classic program within Riemannian and differential geometry to understand how curvature hypotheses combined with a priori mild topological conditions give rise to strong rigidity properties. Among the many classic examples in this program, let us highlight some rigidity properties associated to the PMT of GR. These concern AE manifolds of non-negative scalar curvature, which satisfy decay properties so that the ADM energy
\begin{align}\label{ADMenergy}
E_{ADM}(g)=\frac{1}{2(n-1)\omega_{n-1}}\lim_{r\rightarrow\infty}\int_{S_{r}}\left(\partial_ig_{ij} - \partial_{j}g_{ii} \right)\nu^j_{\delta}d\omega_{\delta}
\end{align}
is well defined. Above, $\nu_{\delta}$ stands for the Euclidean outward pointing unit normal to an Euclidean sphere $S_r\subset \mathbb{R}^n\backslash \overline{B_1(0)}$, while $d\omega_{\delta}$ stands for the volume element induced on $S_r$ by the Euclidean metric. This quantity is a geometric invariant within suitable classes of AE manifolds, and the PMTs state that, under suitable decaying conditions, AE manifolds with non-negative scalar curvature have non-negative ADM energy, and, in particular, $E_{ADM}(g)=0$ iff $(M^n,g)\cong (\mathbb{R}^n,\cdot)$, thus characterising Euclidean space as the unique  AE manifold with zero ADM energy and non-negative scalar curvature. This rigidity of Euclidean space through the PMT is central in many other problems associated to scalar curvature. For instance, it shows that non-negative scalar curvature cannot be localised in a compact set of an AE manifold, which has motivated scalar curvature constructions such as \cite{Carlotto3,Corvino1}. 

One further remarkable rigidity property of AE manifolds is that the Euclidean space is  the only such manifold which is Ricci-flat (see, for instance, \cite{AvalosFreitas,Lee-Parker,MR849427}). This rigidity is central to many uniqueness results, some examples of which can be found in \cite{AvalosFreitas}. Let us notice that this Ricci-flat rigidity of AE manifolds is also linked to the previously mentioned rigidity of the PMT. In fact, classically, the rigidity statement in the PMT uses this fact. Furthermore, it has been shown in \cite{Herzlich,Miao2} that (\ref{ADMenergy}) can be rewritten in terms of the Einstein tensor as
\begin{align}\label{ADMenergy2}
E_{ADM}(g)=-\frac{1}{(n-1)(n-2)\omega_{n-1}}\lim_{r\rightarrow\infty}\int_{S_r}G_{g}(r\partial_r,\nu_{\delta})d\omega_{\delta},
\end{align} 
where $G_g\doteq \mathrm{Ric}_g-\frac{1}{2}R_g\:g$ stands for the Einstein tensor. 

In this paper, we are mainly interested in exploring a fourth-order analogue to this Ricci-flat rigidity property, which is linked to $Q$-curvature analysis. Let us recall that, given a Riemannian manifold $(M^n,g)$, its $Q$-curvature is defined by
\begin{align}
\label{Qgintro}
\begin{split}
Q_g&\doteq -\frac{1}{2(n-1)}\Delta_gR_g - \frac{2}{(n-2)^2}|\mathrm{Ric}_g|^2_g + \frac{n^3-4n^2+16n-16}{8(n-1)^2(n-2)^2}R_g^2,
\end{split}
\end{align}
and this can be seen as a non-linear fourth-order operator on the set $\mathrm{Met}(M)$ of Riemannian metrics on $M$:
\begin{align*}
Q:\mathrm{Met}(M)&\to C^{\infty}(M),\\
g&\mapsto Q_g.
\end{align*}
In this setting, if we denote by $S_2M$ the bundle of symmetric $(0,2)$-tensor fields over $M$, then the linearisation $DQ_g$ of $Q$ at the metric $g$ is given by a map
\begin{align*}
DQ_g:S_2M\to C^{\infty}(M),
\end{align*}
and its formal $L^2$-adjoint is then given by a map $DQ^{*}_g:C^{\infty}(M)\to S_2M$. In this setting, in \cite{aa43}, the authors introduced a $(0,2)$-tensor field canonically associated to $Q$-curvature, given by
\begin{align}\label{J-tensor}
J_g\doteq -\frac{1}{2}DQ^{*}_g(1),
\end{align}
which satisfies a local conservation law
\begin{align}
\label{eqdivJ}
\mathrm{div}_g(J_g-\frac{1}{4}Q_g\:g)=0.
\end{align}

The above relations led the authors in \cite{aa43} to consider $J_g$ as a fourth order analogue of the Ricci tensor, while from \eqref{eqdivJ} one can introduce an analogue of the Einstein tensor:
\begin{align}
G_{J_g}\doteq J_g-\frac{1}{4}Q_g\:g 
\end{align}
as the \emph{$J$-Einstein tensor}. In this setting, Riemannian manifolds for which $\mathrm{Ker}(DQ^{*}_g)\neq \{0\}$ are called \emph{$Q$-singular}. Rigidity of $Q$-singular manifolds has been studied recently, for instance in \cite{aa43,LinYuan1,LinYuan2}. It had been observed by \cite{ChangGursk} that, on a 4-dimensional manifold, $1\in \mathrm{Ker}(DQ^{*}_g)$ iff $g$ is $Q$-flat and Bach-flat. This conclusion can actually be obtained directly from the work of \cite{aa43}, where the authors show that $J_g$ can be rewritten as
\begin{align}\label{J-tensor.2}
J_g=\frac{1}{n}Q_gg-\frac{1}{n-2}B_g -\frac{n-4}{4(n-1)(n-2)}T_g,
\end{align}  
where $B_g$ stands for the Bach tensor and $T_g$ is defined via
\begin{align}\label{T-tensor}
T_g\doteq (n-2)(\nabla^2\mathrm{tr}_gS_g-\frac{1}{n}g\Delta_g\mathrm{tr}_gS_g)+4(n-1)(S_g\times S_g - \frac{1}{n}|S_g|^2_gg)-n^2(\mathrm{tr}_gS_g)\overset{\circ}{S}_g.
\end{align}
Above, $S_g\doteq \frac{1}{n-2}\left(\mathrm{Ric}_g-\frac{1}{2(n-1)}R_g g \right)$ stands for the Schouten tensor; $S_g\times S_g$ stands for the symmetric $(0,2)$-tensor locally defined by $(S_g\times S_g)_{ij}\doteq S^{k}_iS_{kj}$, and $\overset{\circ}{S}_g$ stands for the traceless part of $S_g$.

Recently, in \cite{LinYuan2}, YJ. Lin and W. Yuan studied rigidity phenomena associated to $J$-Einstein metrics on closed manifolds. For instance, under curvature and pinching hypotheses, they establish that $J$-Einstein closed manifolds are actually Einstein.\footnote{Notice that Einstein manifolds are always $J$-Einstein (see, for instance, \cite{aa43}).} Also, a volume comparison theorem based on a $Q$-curvature comparison for closed Einstein manifolds satisfying certain stability properties is a central object of study in \cite{LinYuan2}. 

In the above setting, we will analyse rigidity properties associated to $J$-flat AE manifolds. This analysis is motivated by some of the above results and the rigidity of Ricci-flat AE manifolds. In particular, our main result will be the following rigidity statement:

\begin{theorem}\label{J-flatnessRigidity}
Let $(M^n,g)$,  be a smooth $W^{3,p}_{-\tau}$ AE manifold with $p>n \ge 3$ and $\tau>0$. If $J_g=0$ and $Y([g])>0$, then $(M^n,g)\cong (\mathbb{R}^n,\cdot)$.
\end{theorem}

Above, the $W^{k,p}_{\delta}$ Sobolev spaces are used to control the decay at infinity of $g$ and we refer the reader to Definition \ref{WeightedSob1} and Definition \ref{AEManifolds2} for further details. Also, in the above theorem we have imposed the condition $Y([g])>0$, which means that the AE manifold $(M^n,g)$ is \emph{Yamabe positive}. In this setting, the Yamabe invariant is defined as the following conformal invariant
\begin{align}\label{YamabeInvariant}
Y([g])\doteq \inf_{u\in C^{\infty}_0(M)}\frac{\int_{M}(a_n|\nabla u|^2_g + R_gu^2)dV_g}{\Vert u\Vert^2_{L^{\frac{2n}{n-2}}}},
\end{align}
where $a_n\doteq \frac{4(n-1)}{n-2}$ and $[g]$ denotes the conformal class of $g$. For detailed properties of the Yamabe problem on AE manifolds, we refer the reader to \cite{Maxwell-Dilts}.


It should be noted that, while we use decays controlled by weighted Sobolev spaces as they are well adapted to the methods we employ, Theorem \ref{J-flatnessRigidity} covers classical hypotheses involving pointwise decays on the metric. Notably, if  $\vert g - \delta \vert =O_3(r^{-\tau})$ for some $\tau >0$ near infinity, $J_g=0$ and $Y([g])>0$, we naturally have that $(M^n,g)$ is $W^{3,p}_{-\tau'}$ for some $0<\tau'<\tau$, which allows one to apply Theorem \ref{J-flatnessRigidity}. In particular the following corollary stands:
\begin{corollary}
Let $(M^n,g)$ be  a Riemannian manifold such that $\vert g - \delta \vert =O_3(r^{-\tau})$  near infinity, for $\tau>0$.  If $J_g=0$ and $Y([g])>0$, then $(M^n,g)\cong (\mathbb{R}^n,\cdot)$.
\end{corollary}

In addition, the $W^{3,p}_{-\tau}$ condition in the above theorem means that, a priori, we only control the decay of  $g$ and its derivatives up to third order. This is quite remarkable, since  the geometric objects involved in this theorem are fourth order (and therefore one will need to control fourth order derivatives as well).
Thus, a crucial step in the proof of the above theorem will be to show that the $J$-flat condition actually provides us with additional control for the decays of higher order derivatives. Furthermore, we will need to show that we can bootstrap the order of decay $\tau$ from being merely positive to actually achieving $\tau>\frac{n-4}{2}$. This lower bound for $\tau$ will be needed to appeal to the rigidity of the positive energy theorem associated to a fourth order invariant related to $Q$-curvature. This invariant was introduced as a canonically associated conserved quantity in a family of fourth order gravitational theories in \cite{avalos2021energy}, which, in the case of stationary space-time solutions reduces to 
\begin{align}\label{FourthOrderEnergy}
\mathcal{E}(g)=\lim_{r\rightarrow\infty}\int_{S_r}\left( \partial_j\partial_i\partial_ig_{aa} - \partial_j\partial_a\partial_ig_{ai}\right)\nu^j_{\delta}d\omega_{\delta}.
\end{align}

In \cite[Theorem A]{avalos2021positive}, the above notion of energy was analysed in detail and the following positive energy theorem was established
\begin{theo}[Positive Energy \cite{avalos2021positive}]\label{PETHM}
Let $(M^n,g)$ be an $n$-dimensional AE manifold, with $n\geq 3$, which satisfies the decaying conditions: (i) $g_{ij}-\delta_{ij}=O_4(r^{-\tau})$, with $\tau>\max\{0,\frac{n-4}{2}\}$,  in some coordinate system associated to a structure of infinity; (ii) $Q_g\in L^1(M,dV_g)$, and such that $Y([g])>0$ and $Q_g\geq 0$. Then, the fourth order energy $\mathcal{E}(g)$ is non-negative and $\mathcal{E}(g)=0$ if and only if $(M,g)$ is isometric to $(\mathbb{R}^n,\delta)$.
\end{theo}
 
In order to appeal to Theorem \ref{PETHM} within the proof of Theorem \ref{J-flatnessRigidity}, we will first establish an intermediary result, proving that (\ref{FourthOrderEnergy}) can be rewritten in terms of the J-Einstein tensor at infinity, in an analogue way to how the ADM energy (\ref{ADMenergy}) can be rewritten in terms of the Einstein tensor in (\ref{ADMenergy2}). For this we will follow the ideas of \cite{Herzlich}, but now identifying the total $Q$-curvature as the natural Langrangian out of which we can build asymptotic geometric invariants following \cite{Michel}. In particular, in Theorem \ref{HerzlichThm} we establish the following result:
\begin{theorem}\label{HerzlichThmIntro}
Let $(M^n,g)$ be an AE manifold of order $\tau>\max\{0,\frac{n-4}{2}\}$ satisfying $Q_g\in L^1(M,dV_g)$. Then, the following identity holds:
\begin{align}
\frac{n-4}{8(n-1)}  \mathcal{E}(g)=-\lim_{r\rightarrow\infty}\int_{S_r}G_{J_g}(X,\nu_{\delta})d\omega_{\delta},
\end{align}
where $X=r\partial_r$. In particular, the limit in the right-hand side exists and is finite. 
\end{theorem}

As written, Theorem \ref{HerzlichThmIntro} stands for all $n\ge 3$ dimensional manifolds, but is only insightful for $n \ge 5$. Indeed, as was highlighted in \cite{avalos2021positive}, four dimensional AE metrics have zero mass, while in \cite{avalos2021energy}, it was shown that to have non zero mass, three dimensional AE metrics must have growth at infinity. Imposing decay on the metric (and thus also on $G_{J_g}$) then forces  $\mathcal{E}(g)=0 = \lim_{r\rightarrow\infty}\int_{S_r}G_{J_g}(X,\nu_{\delta})d\omega_{\delta}$ for $n=3,4$.

Theorem \ref{HerzlichThmIntro} also represents the exact analogue of the results in \cite{Herzlich} in the AE setting and in the context of $Q$-curvature invariants. This provides a direct link between the $J$-tensor at infinity and the fourth order energy (\ref{FourthOrderEnergy}), which from \cite{avalos2021positive} we know to be related to several rigidity phenomena associated to $Q$-curvature. For instance, $\mathcal{E}(g)$ is positively proportional to the mass of the Paneitz operator studied in \cite{Raulot,Malchiodi1,hangyang2}. Most importantly for our purposes, we directly see that any $J$-flat metric which satisfies the conditions of Theorem \ref{HerzlichThmIntro} must have zero energy. Therefore, after establishing this result, the work leading up to Theorem \ref{J-flatnessRigidity} will consist in proving that the $J$-flatness condition guarantees that both the hypotheses of Theorems \ref{HerzlichThmIntro} and Theorem \ref{PETHM} are satisfied. In doing so, we shall prove some fourth order analogues to classical results known from \cite{MR849427}. In particular, we shall establish the following theorem, which proves that the $J$-tensor controls the optimal decay for the metric in asymptotic harmonic coordinates. {In particular, one sees that an a priori control on $J$ can be used to increase both the number of derivatives decaying at infinity, and the rate of their decay. Notably, we may start with only three derivatives controlled in weighted Sobolev spaces, and, if $J_g$ remains in some $W^{k,p}_{-\delta}$-space, we obtain control for the intermediary derivatives of order $l$, with $3\leq l\leq 3+k$.\footnote{{See Remark \ref{DecayBootsrapRemark} for further comments on the optimality of this result.}} More precisely:}
\begin{theorem}\label{AEDegree-JtensorIntro}
Suppose $(M^n,g)$ is an AE manifold of class $W^{3,p}_{-\tau}(M,\Phi)$, {$p>n\geq 3$}, such that $J_g \in W^{k-4,p}_{-\delta-4}(M,\Phi)$, for an integer $k \ge 4$ {and a real number $\delta\geq \tau$}, with respect to some structure of infinity $\Phi$. Then, in harmonic coordinates at infinity given by the chart $\Theta$,  $(M^n,g)$ is of  class $W^{k,p}_{-\tau}(M,\Theta)$, and if $\tau< \delta<n-4$, then $(M^n,g)$ is of class $W^{k,p}_{-\delta}(M,\Theta)$.
\end{theorem}

Before finishing this introduction, let us highlight how Theorems \ref{J-flatnessRigidity} and \ref{AEDegree-JtensorIntro} prove that the fourth order $J$-tensor retains crucial controls of the geometry of an AE manifold, which are known to be provided by the Ricci tensor. This might be unexpected due to the intricate and much more involved definition of $J_g$, which could have implied losses on these controls and, for instance, much richer structure of $J$-flat AE manifolds. On the contrary, we still see that this condition is highly rigid (at least on the Yamabe positive side) and that from the decay of $J$ at infinity one can still directly read out the optimal decay of $g$ in harmonic coordinates. {Let us also highlight that they are more subtle than 
in the second order case, something which can be quite explicitly seen within the proof of Theorem \ref{AEDegree-JtensorIntro}. This, once more, comes down to the intricate definition of the $J$-tensor and the higher order nature of the problem, which demands us to also provide some non-trivial generalisations of second order results.} {In particular, Theorem \ref{AEDegree-JtensorIntro} requires understanding how changing structures of infinity affects weighted Sobolev spaces. We develop ideas in R. Bartnik's \cite{MR849427} to highlight that in our framework, one can switch to harmonic coordinates without affecting the regularity or decay of the involved tensor, up to the orders required for the proofs (see Theorem \ref{HarmonicCoordThm}).  }
 
With the above results in mind, the paper is organised as follows. {In Section 2, we shall review the necessary definitions associated to AE manifolds and introduce some notational conventions. We shall establish a few results which follow along standard ones, but which are tailored for the application in this paper. Then, Section 3 is devoted to the study of the Sobolev spaces and their behaviour under switches to harmonic coordinates, while Section 4 focuses on the proof of Theorem \ref{HerzlichThmIntro} and Section 5 on the proof of Theorems \ref{J-flatnessRigidity} and \ref{AEDegree-JtensorIntro}.}

\bigskip
{\bf Acknowledgments:} The authors would like to thank the CAPES-COFECUB and CAPES/MATH-AmSud for their financial support. Also, Rodrigo Avalos would like to thank FUNCAP and the Alexander von Humboldt Foundation for their financial support and Paul Laurain we would like thank ANR (ANR- 18-CE40-002) for their financial support. Finally, we would like to thank professor Jorge H. Lira for helpful discussions related to this project and the problems treated in this paper, { as well as Melanie Graf for helpful remarks concerning section 3 of this paper.} 

\section{Weighted spaces on AE manifolds}
We will first recall the definitions for weighted Lebesgue and Sobolev spaces, using R. Bartnik's notations  \cite[Definition 1.1]{MR849427}
\begin{de}\label{WeightedSob1}
The weighted Lebesgue spaces $L^p_\delta$, $1 \le p \le \infty$, $\delta \in \R$ are the set of measurable functions in $L^p_{\mathrm{loc}}\left(\R^n\right)$ such that the norms $\| \cdot \|_{p,\delta}$ defined by 
$$\begin{aligned} &\| u \|_{p,\delta} = \left\| u \sigma^{-\delta - \frac{n}{p} } \right\|_{L^p(\R^n)} \text{ if } p<+ \infty \\
&\| u \|_{\infty,\delta} = \left\| u \sigma^{-\delta } \right\|_{L^\infty(\R^n)},
\end{aligned}$$
are finite. Here $ \sigma \doteq \sqrt{1+r^2}$
{with $r(x)\doteq |x|$}.

The weighted Sobolev spaces $W^{k,p}_\delta$ are then defined in a similar manner as the set of measurable functions in $W^{k,p}_{\mathrm{loc}}\left( \R^n \right)$ proceeding from the $\|\cdot \|_{k,p,\delta}$ norms:
$$ \|u\|_{k,p,\delta}= \sum_{j=0}^k \| \nabla^j u \|_{p, \delta -j}.$$
\end{de}

\begin{de}\label{AEmanifolds}
Let $(M,g)$ be a complete, smooth, connected, $n$ dimensional  Riemannian manifold  and let $\tau>0$. We say that $(M,g)$ is an Asymptotically Euclidean (AE) manifold of class $W^{k,p}_{-\tau}$  if:
\begin{enumerate}
\item There exists a finite collection $(E_i)_{i=1}^m$ of open subsets of $M$ and diffeomorphisms $\Phi_i : \, E_i  \rightarrow   \R^n\backslash \overline{B_1(0)} $ such 
that  $M \backslash \cup_i E_i$ is compact.
\item For each integer $1\le i \le m$, $ 1\le a,b \le n$,  $\left(\left({\Phi_i^{-1}}\right)^* g \right)_{ab} - \delta_{ab} \in W^{k,p}_{-\tau} (\R^n\backslash \overline{B_1(0)})$.
\end{enumerate}
\end{de}

It should be noticed that, since we assume our manifolds $(M,g)$ to be smooth, saying they are AE of class $W^{k,p}_{-\tau}$ does not impact the differentiability of the metric, but up to what order they can be assumed to decay at infinity.

\begin{remark}
During the core of the paper, in our main theorems, we will  consider AE manifolds with one end. The necessary modifications for the general case are quite straightforward from well-known arguments and an appeal to the same tools we shall develop. In particular, the core of the analysis within our proofs is done by working directly within each end of the manifold, thus allowing us to localise the problem. The main necessary modifications relate to the definition of the fourth order energy obtaining contributions from each end separately, with the total energy being non-negative and the critical case being rigid.
\end{remark}

\begin{de}\label{AEManifolds2}
{The charts $\Phi_i$ are called end charts, and the corresponding coordinates are end coordinates.  Given $(M,g)$ AE  
and $\{U_i\doteq E_i, \Phi_i \}_{i=1}^m$ the collection of end charts, $K \doteq M \backslash \cup_i E_i$ is a compact manifold. Let $\{U_i,\Phi_i\}_{i=m+1}^N$ be a finite number of coordinate charts covering the compact region $K$. We can then consider a partition of unity $\{\eta_i\}_{i=1}^N$ subordinate to the coordinate cover $\{U_i,\Phi_i\}_{i=1}^N$. Then, given a vector bundle $E\xrightarrow{\pi} M$, we define $W^{k,p}_{\delta}(M;E)$ to be the subset of $W^{k,p}_{loc}(M;E)$ such that 
\begin{align}\label{GlobalWeightedAENorm}
\Vert u\Vert_{W^{k,p}_{\delta}}&\doteq\sum_{i=1}^{m}\Vert {\Phi^{-1}_i}^{*}(\eta_i u)\Vert_{W^{k,p}_{\delta}(\mathbb{R}^n)} + \sum_{i=m+1}^N\Vert{\Phi^{-1}_i}^{*}(\eta_i u)\Vert_{W^{k,p}(U_i)}<\infty. 
\end{align}}
\end{de}

\begin{remark}
{Let us recall from \cite{MR849427} that the $L^p_{\delta}$-spaces are actually intrinsic to $M$ and hence independent of any specific structure of infinity. Thus, we shall write $L^p_{\delta}(M)$ without reference to a chosen structure of infinity without causing ambiguity. A priori, the higher order weighted Sobolev spaces do depend on the chosen end charts, and thus when any ambiguity can occur, we shall make explicit reference to the chosen structure. Since we will work with only one end we will write $W^{k,p}_{\delta}(M,\Phi)$ for a given structure of infinity $\Phi$. When we deal with only one fixed structure of infinity, we shall avoid such heavier notation and simply write $W^{k,p}_\delta(M)$.} 
\end{remark}

Working separately on each end, we can extend the classical properties of weighted Sobolev spaces in $\R^n$ to spaces in $M$ \cite[Theorem 1.2]{MR849427}. We will notably use  the algebraic properties of the $W^{k,p}_{\delta}$ spaces for $p>n$. The following multiplication property is well-known from standard literature.
\begin{lem}\label{MultiplicationPropertyGeneral}
Let $V\rightarrow M^n$ be a vector bundle over $M$. 
If $1<p\leq q<\infty $ and $k_1+k_2>\frac{n}{q}+k$ where $k_1,k_2\geq k$ are non-negative integers, then, we have a continuous multiplication property $W^{k_1,p}_{\delta_1}\otimes W^{k_2,q}_{\delta_2}\to W^{k,p}_{\delta}$ for any $\delta>\delta_1+\delta_2$. In particular, $W^{k,p}_{\delta}$ is an algebra under multiplication for $k>\frac{n}{p}$ and $\delta<0$.
\end{lem}

The above multiplication property can be deduced using the tools developed in \cite{MR849427},  and can also be found in \cite[Lemma 5.5]{Cantor-SplittingTensors} and the corresponding $L^2$-version can also be found in \cite[Lemma 2.5]{CB-C}.\footnote{Some of these proofs are done for scalar functions in the cited references. The adaptation for the case of vector valued functions follows from well-known localisation arguments.} Below, we will now establish a multiplication property which is tailored for some of our specific applications.

\begin{prop}\label{Wkpdeltaalgebra}
If $p>n$, $k_1,k_2\ge 1$, $f \in W^{k_1,p}_{\delta_1}(M)$ and $g\in W^{k_2,p}_{\delta_2}(M)$, then  $fg \in W^{k,p}_{\delta_1+\delta_2}(M)$ with $k=\min\{k_1,k_2\}$.
\end{prop}
\begin{proof}
If $k=1$, since $p>n$, Sobolev embeddings ensure that there exists $C$ such that $|f|\le C \sigma^{\delta_1}$, $|g|\le C\sigma^{\delta_2}$.
Then, $fg \in L^p_{\delta_1+\delta_2}$ and  $$ \begin{aligned}
\int \left| \nabla \left[ fg \right] \right|^p \sigma^{-\delta_1 p-\delta_2p+p -n } & \le \int \left|\nabla f g + f \nabla g \right|^p \sigma^{-\delta_1 p+p-\delta_2p-n } \\
&\le C_p \int  \left( |\nabla f |^p |g|^p + |f|^p |\nabla g|^p \right) \sigma^{-\delta_1 p-\delta_2p+p-n } \\
&\le C_p \int  | \nabla f|^p \sigma^{-\delta_1p +p-n} \left(|g|\sigma^{-\delta_2} \right)^p + |\nabla g |^p \sigma^{-\delta_2 p +p-n} \left( |f| \sigma^{-\delta_1} \right)^p  \\
& \le C_p \left( \| \nabla f\|^p_{L^p_{\delta_1-1}}+ \|\nabla g \|^p_{L^p_{\delta_2-1}} \right).
\end{aligned}$$
Hence $fg \in W^{1,p}_{\delta_1+\delta_2}(M).$

For $ k >1$, the same reasoning with the Leibniz formula yields the result.
\end{proof}

We will also use the Sovolev embeddings \cite[Theorem 1.2, iv), iv)]{MR849427} which we recall here:

\begin{theo}
\begin{itemize}
\item If $u \in W^{k,p}_\delta$, then
$$\| u \|_{L^{\frac{np}{n-kp}}_\delta} \le C\|u\|_{W^{k,p}_\delta}$$
if $n-kp>0$ and $p \le q \le \frac{np}{n-kp}$,
$$\| u \|_{L^{\infty}_\delta} \le C\|u\|_{W^{k,p}_\delta}$$
if $n-kp<0$, and in fact 
$$|u(x)|= o \left( r^\delta\right)$$
as $r\rightarrow \infty.$
\item If $u \in W^{k,p}_\delta$, $0< \alpha \le k - \frac{n}{p} \le 1$, then 
$$\|u \|_{C^{0,\alpha}_\delta} \le C\|u \|_{W^{k,p}_\delta},$$
where the weighted Hölder norm is defined by
$$\| u\|_{C^{0,\alpha}_\delta} = \sup \left( \sigma^{-\delta + \alpha}(x) \sup_{4|x-y|\le \sigma(x)} \frac{ |u(x)-u(y)|}{|x-y|^\alpha }\right) + \sup \left( \sigma^{-\delta} |u(x)| \right).$$
\end{itemize}
\end{theo}

\begin{remark}
We can define the $C^{k,\alpha}_\delta$ spaces as the set of functions for which the norm
$$\| u\|_{C^{k,\alpha}_\delta} = \sum_{j=0}^k \| \nabla^j u \|_{C^{0,\alpha}_{\delta - j}}$$
is finite, while the $C^k_\delta$ spaces are obtained with 
$$\| u\|_{C^{k}_\delta} = \sum_{j=0}^k \| \nabla^j u \|_{C^{0}_{\delta - j}}$$
where the weighted $0$ norm is
$$\| u\|_{C^{0}_\delta} =  \sup \left( \sigma^{-\delta} |u(x)| \right).$$
\end{remark}

A key step in our proof of Theorem \ref{J-flatnessRigidity} will be to use elliptic regularity to improve the decays. {In particular, we shall appeal to an adaptation of results appearing in L. Nirenberg and H. F. Walker \cite{DouglisNirenberg} as well as R. Bartnik \cite{MR849427}. Since we will use it later with several concomitant structures at infinity, we will highlight the end charts in its statement.}

\begin{lem}\label{lemregularite}
Let $(M^n,g)$ be an AE manifold of class $W^{k-1,p}_{-\tau}$, 
$k\geq 2$ and $p>n$. If a function $f$ on $M$ satisfies $\Delta_g f \in W^{k-2,p}_{-\delta-2}(M, \Phi)$  and $ f \in L^p_{-\delta}(M, \Phi)$ {with $\delta\in\mathbb{R} $}, then $f \in W^{k,p}_{-\delta}(M, \Phi).$
\end{lem}

\begin{proof}

First, since the coefficients of $\Delta_g$ are smooth, then, for any bounded domain $\Omega\subset M$ with smooth boundary, $\Delta_g:C^{\infty}(\Omega) \to C^{-\infty}(\Omega)$, and usual local elliptic regularity applies to show that if $\Delta_gf\in W^{k-2,p}(\Omega)\Longrightarrow f\in W^{k,p}(\Omega)$. Thus, under our hypotheses, we know that $f\in W^{k,p}_{loc}(M)\cap L^p_{-\delta}(M)$. Since we need only provide control at infinity now, using an appropriate cut-off function, we can assume $f$ to be supported in a neighbourhood of infinity in one end of $M$. That is, without loss of generality, we may assume for our purposes that $\mathrm{supp}(f\circ\Phi^{-1})\subset \mathbb{R}^n\backslash \overline{B_{R_0}(0)}$ for some ${R_0}>1$ sufficiently large. Then, our claim reduces to showing that $f\circ\Phi^{-1}\in W^{k,p}_{-\delta}(\mathbb{R}^n\backslash\overline{B_{R_0}(0)})$. We shall divide this proof in two steps. The first one consists in showing the case $k=2$ and the second one is an induction for $k>2$.  {The case $k=2$ is close to several statements in the literature. In particular,
 it follows from mild adaptations of \cite[Theorem 3.1]{DouglisNirenberg} to our setting. As pointed out in \cite[Theorem 4.1]{CB-C} and the remark following it, the results of \cite[Theorem 3.1]{DouglisNirenberg} hold under the weaker assumption $u\in W^{m,p}_{loc}$. Then, the same method of proof as in \cite[Theorem 3.1]{DouglisNirenberg} applies to establish our $k=2$ case.}\footnote{{In this case condition $(ii)$ for lower order coefficients in \cite[Theorem 3.1]{DouglisNirenberg} gets replaced by Sobolev controls. An appeal to Lemma \ref{MultiplicationPropertyGeneral} provides control over the corresponding terms in the equation, whose coefficients for $R_0$ sufficiently large will have Sobolev norm arbitrarily small.}}

Let us now consider $k>2$ and work inductively, assuming that the lemma stands for $k-1$. 
 Let $g-\delta \in W^{k-1,p}_{-\tau}(\mathbb{R}^n\backslash\overline{B_1(0)}))$ and $f $ be such that  $\Delta_g f \in W^{k-2,p}_{-\delta-2}(M)$ and $ f \in L^p_{-\delta}(\mathbb{R}^n\backslash\overline{B_1(0)}))$. Since $W^{k-2,p}_{-\delta -2}(\mathbb{R}^n\backslash\overline{B_1(0)}) \subset W^{k-3,p}_{-\delta-2}(\mathbb{R}^n\backslash\overline{B_1(0)})$ and $W^{k-1,p}_{-\tau} \subset W^{k-2,p}_{-\tau}(\mathbb{R}^n\backslash\overline{B_1(0)})$, the lemma for $k-1$ ensures that $f \in W^{k-1,p}_{-\delta}(\mathbb{R}^n\backslash\overline{B_1(0)})$. Then, for any integer $1\le l \le n$:
$$\begin{aligned}
\Delta_g \left( \partial_l f \right) = \partial_l \left(\Delta_g f \right) -\partial_l g^{ij} \partial_{ij}f + \partial_l \left( g^{ij}\Gamma^k_{ij}\right) \partial_k f.
\end{aligned}$$
Since $f$ has been shown to be in $W^{k-1,p}_{-\delta}(\mathbb{R}^n\backslash\overline{B_1(0)})$ and $g-\delta \in W^{k-1,p}_{-\tau}(\mathbb{R}^n\backslash\overline{B_1(0)})$, appealing to Lemma \ref{MultiplicationPropertyGeneral} we find $\partial_l g^{ij}\partial_{ij}f \in W^{k-3,p}_{-\delta -3}(\mathbb{R}^n\backslash\overline{B_1(0)}))$. Similarly, since $g^{ij} \Gamma^k_{ij} \in W^{k-2,p}_{-\tau - 1}(\mathbb{R}^n\backslash\overline{B_1(0)})$, $\partial_l \left( g^{ij}\Gamma^k_{ij}\right) \partial_k f \in  W^{k-3,p}_{-\delta-3}(\mathbb{R}^n\backslash\overline{B_1(0)}).$  
Finally, since by assumption $\Delta_g f \in W^{k-2,p}_{-\delta-2}(\mathbb{R}^n\backslash\overline{B_1(0)}))$, $\partial_l \Delta_g f \in W^{k-3,p}_{-\delta-3}(\mathbb{R}^n\backslash\overline{B_1(0)}))$.  Combining all these considerations yields:
$$\Delta_g \left( \partial_l f \right) \in W^{k-3,p}_{-\delta-3}(\mathbb{R}^n\backslash\overline{B_1(0)}).$$
Applying once more the lemma with $k-1$ then yields that $\partial_l f \in W^{k-1,p}_{-\delta-1}(\mathbb{R}^n\backslash\overline{B_1(0)})$. This stands true for all partial derivatives $\partial_l f$ and since $f \in L^p_{-\delta}(\mathbb{R}^n\backslash\overline{B_1(0)}))$, it ensures that $f \in W^{k,p}_{-\delta }(\mathbb{R}^n\backslash\overline{B_1(0)})$, which concludes the proof.

\end{proof}

As was mentioned in the introduction, we will also need to improve the order of decay. A key point in our proof will thus be the  following regularity result for weighted spaces \cite[Proposition 1]{MR2116728} (which extends \cite[Proposition 1.14]{MR849427} in the same manner as lemma \ref{lemregularite} follows from \cite{MR849427}[Prop 1.6]):

\begin{prop}
\label{regularitymaxwell}
Assume that $(M,g)$ is AE of class $W^{k,p}_{-\tau}$, $k \ge 2$ and $p> \frac{n}{k}$. Then, if $0<\delta < n-2$ the operator $\Delta_g: \, W^{k,p}_{-\delta} (M) \rightarrow W^{k-2,p}_{-\delta -2} (M)$ is Fredholm of index $0$. In particular, if $\tau \le \delta$ and $X \in W^{2,p}_{-\tau}(M)$, $\Delta_g X \in W^{k-2,p}_{-\delta-2}(M)$ then $ X \in W^{k,p}_{-\delta}(M).$
\end{prop}

\section{Harmonic end coordinates}

{In this subsection we shall  {further develop} some concepts from R. Bartnik \cite{MR849427} associated to harmonic structures of infinity to higher orders of regularity. Harmonic coordinates are special charts at infinity conceived to have optimal regularity and decay properties, and for which the Ricci tensor can be expressed as an elliptic operator on the metric:
\begin{align}\label{RicciHarmonic}
{\mathrm{Ric}_g}_{ij}=-\frac{1}{2}g^{ab}\partial_{ab}g_{ij} + f_{ij}(g,\partial g),
\end{align}
where $f_{ij}$ are smooth functions on their arguments, which are in particular quadratic polynomials on $\partial g$. }

{A key idea in our proof will be to use the fourth order tensorial equalities to gain regularity on the Ricci tensor, and then to apply \eqref{RicciHarmonic} to transfer this regularity on the metric after switching to harmonic coordinates. This requires showing that no Sobolev control on the Ricci tensor is lost during the switch. While this may seem natural, it forces us to understand how weighted Sobolev controls are affected under changes of structures of infinity. We notice that this question arises in the context of classic Sobolev spaces as well, where related results are \emph{standard}. That is, let $\Omega$ be a domain in $\mathbb{R}^n$ and consider a diffeomorphism $\Phi:\Omega\to\Omega'$. The invariance of the Sobolev spaces $W^{k,p}(\Omega)$ under such coordinate change is known to depend on the $C^k$-control of $\Phi$ and $\Phi^{-1}$. In particular, $C^{\infty}$-differomorphisms on bounded domains always induce equivalent Sobolev norms, which is the key fact behind the invariance of Sobolev spaces on compact manifolds. In the case $\Omega$ is unbounded, then such $C^k$-controls have to be imposed as additional hypothesis on the admissible coordinate changes (see, for instance, \cite[Theorem 3.41]{Adams}).}


{The core of the difficulty in the above claims lies in the chain rule. For instance, if $f(x) \in W^{1,p}(\Omega)$ and $\Phi$ is a diffeomorphism at infinity for the change of coordinates $\Phi(x)=y$, it is not enough to show that $(\nabla f)\circ \Phi^{-1} \in L^p(\Omega)$ to conclude that $f\circ \Phi^{-1} \in W^{1,p}(\Omega)$. On the contrary, $\nabla\left( f \circ \Phi^{-1} \right) = (\nabla f )\circ \Phi^{-1} \nabla \Phi^{-1}$, which highlights the necessity  for strong controls on the inverse of the change of coordinates. Carrying this to any order, for weighted spaces, {and only for a priori controls on $\Phi$} are the difficulties one is faced with in the following results. We show that for a switch to a set of harmonic coordinates, these controls come naturally.} {Related results with application to GR can be consulted in \cite{Graf}.}


With all the above in mind, let us consider two real numbers $R_i>0$, $i=1,2$, and a diffeomorphism
\begin{align*}
\Phi:\mathbb{R}^n\backslash\overline{B_{R_1}(0)} \to \mathbb{R}^n\backslash\overline{B_{R_2}(0)}.
\end{align*}
Consider then the operator $A:\mathcal{L}(\mathbb{R}^n\backslash\overline{B_{R_1}(0)})\to \mathcal{L}(\mathbb{R}^n\backslash\overline{B_{R_2}(0)})$, where $\mathcal{L}(\Omega)$ denotes the set of Lebesgue measurable functions on the domain $\Omega\subset \mathbb{R}^n$, given by:
\begin{align}\label{CoordinateChangeOp}
(Au)(y)\doteq u(\Phi^{-1}(y)),
\end{align}
which is the operator inducing our coordinate change. Below, we shall refer to $y=\Phi(x)$, $x\in \mathbb{R}^n\backslash\overline{B_{R_1}(0)}$ as the coordinates on $\mathbb{R}^n\backslash\overline{B_{R_2}(0)}$. Then, the following lemma holds:
\begin{lem}\label{ChangeOfCoordinatesLemma}
Consider the setting described above for a coordinate change. If $\Phi-\mathrm{Id}\in C_{1-\kappa}^{m}(\mathbb{R}^n\backslash \overline{B_{R_1}(0)})$ and $\Phi^{-1}-\mathrm{Id}\in C_{1-\kappa}^{m}(\mathbb{R}^n\backslash \overline{B_{R_2}(0)})$ for  $m\geq 1$ and some $\kappa>0$, then, given $\delta\in \mathbb{R}$ and $1<p<\infty$, the operator $A$ given by (\ref{CoordinateChangeOp}) maps $W^{m,p}_{\delta}(\mathbb{R}^n\backslash \overline{B_{R_1}(0)})$ continuously onto $W_{\delta}^{m,p}(\mathbb{R}^n\backslash \overline{B_{R_2}(0)})$. 
\end{lem}
\begin{proof}
In the case of $L^p_{\delta}$ the claim follows from the coordinate change formula and the boundedness of $d\Phi$. Since the computation is similar to that of higher derivatives, we shall only work with the latter. Thus, first notice that, for instance from \cite[Theorem 3.41]{Adams}, we know that the weak derivatives of $Af$ of order $1\leq |\alpha|\leq m$ are given by the usual chain rule:
\begin{align*}
\partial^{\alpha}(Af)(y)=\sum^{|\alpha|}_{|\beta|=1}M_{\alpha\beta}(\Phi^{-1})A(\partial^{\beta}f)(y),
\end{align*}
where $M_{\alpha\beta}$ is a polynomial of degree between one and $|\beta|$ in the derivatives of order between one and $|\alpha|$ of the different components of $\Phi^{-1}$. Thus,
\begin{align*}
\int_{\mathbb{R}^n\backslash\overline{B_{R_2}(0)}}|\partial^{\alpha}_y(Af)(y)|^p|y|^{-\delta p+ |\alpha|p -n}dy&\lesssim \sup_{1\leq |\beta|\leq |\alpha|}\sup_{y\in \mathbb{R}^n\backslash\overline{B_{R_2}(0)}}|M_{\alpha\beta}(\Phi^{-1})|\times \\
&\int_{\mathbb{R}^n\backslash\overline{B_{R_2}(0)}}|(\partial^{\alpha}f)(\Phi^{-1}(y))|^p|y|^{-\delta p+ |\alpha|p -n}dy,\\
&\lesssim \Vert d\Phi^{-1}\Vert_{C^{|\alpha|-1}(\mathbb{R}^n\backslash\overline{B_{R_2}(0)})}\times\\
&\int_{\mathbb{R}^n\backslash\overline{B_{R_1}(0)}}|(\partial^{\alpha}f)(x)|^p|x|^{-\delta p+ |\alpha|p -n}|\mathrm{det}(d\Phi)(x)|dx,\\
&\leq C_{\alpha}(d\Phi^{-1})\int_{\mathbb{R}^n\backslash\overline{B_{R_1}(0)}}|(\partial^{\alpha}f)(x)|^p|x|^{-\delta p+ |\alpha|p -n}dx,\\\
&\leq C_{\alpha}(d\Phi^{-1})\Vert \partial^{\alpha}f\Vert_{L^{p}_{\delta-|\alpha|}(\mathbb{R}^n\backslash\overline{B_{R_1}(0)})}
\end{align*}
where the constant $C_{\alpha}$ depends only on $\Vert d\Phi^{-1}\Vert_{C^{|\alpha|-1}(\Omega_2)}$ and fixed parameters. Therefore, 
\begin{align*}
\partial^{\alpha}(Af)\in L^p_{\delta-|\alpha|}(\mathbb{R}^n\backslash\overline{B_{R_2}(0)})\Longleftrightarrow \partial^{\alpha}f\in L^p_{\delta-|\alpha|}(\mathbb{R}^n\backslash\overline{B_{R_1}(0)}),
\end{align*}
which implies 
\begin{align*}
(Af)\in W^{m,p}_{\delta}(\mathbb{R}^n\backslash\overline{B_{R_2}(0)})\Longleftrightarrow f\in W^{m,p}_{\delta}(\mathbb{R}^n\backslash\overline{B_{R_1}(0)}),
\end{align*}
Furthermore, there is a constant $C>0$ such that $\Vert Af\Vert_{W^{m,p}_{\delta}(\mathbb{R}^n\backslash\overline{B_{R_2}(0)})}\leq C \Vert f\Vert_{W^{m,p}_{\delta}(\mathbb{R}^n\backslash\overline{B_{R_1}(0)})}$. The converse implication follows along the same lines, but now appealing to $\Vert d\Phi\Vert_{C^{m-1}}<\infty$.
\end{proof}

\begin{lem}\label{ChangeOfCoordinatesHolderLemma}
Consider the setting described above for a coordinate change. Then, given $\kappa>0$, $m\geq 1$ and $\delta\in \mathbb{R}$, the following statements hold:
\begin{enumerate}
\item If $\Phi-\mathrm{Id}\in C_{1-\kappa}^{0}(\mathbb{R}^n\backslash \overline{B_{R_1}(0)})$, then
\begin{align}\label{A-HolderProp.1}
A:C_{\delta}^{0}(\mathbb{R}^n\backslash \overline{B_{R_1}(0)})\to C_{\delta}^{0}(\mathbb{R}^n\backslash \overline{B_{R_2}(0)}) \text { is bounded }.
\end{align}
\item  If $\Phi-\mathrm{Id}\in C_{1-\kappa}^{m}(\mathbb{R}^n\backslash \overline{B_{R_1}(0)})$, then $\Phi^{-1}-\mathrm{Id}\in C_{1-\kappa}^{m}(\mathbb{R}^n\backslash \overline{B_{R_2}(0)})$ and moreover the operator
\begin{align}\label{A-HolderProp.2}
A:C_{\delta}^{m}(\mathbb{R}^n\backslash \overline{B_{R_1}(0)})\to C_{\delta}^{m}(\mathbb{R}^n\backslash \overline{B_{R_2}(0)}) \text{ is bounded }.
\end{align}
\end{enumerate}
\end{lem}
\begin{proof}
In order to establish the first item, notice the conditions $\Phi-\mathrm{Id}\in C^{0}_{1-\kappa}(\mathbb{R}^n\backslash \overline{B_{R_1}(0)})$ 
implies the existence of a constant $C>0$ such that
\begin{align}\label{EquivDistances.1}
\frac{1}{C} |x| \leq|y(x)|=|\Phi(x)|\leq C|x|.
\end{align}
Thus, for any $x\in \mathbb{R}^n\backslash\overline{B_{R_1}(0)}$ and $f\in  C^{0}_{\delta}(\mathbb{R}^n\backslash \overline{B_{R_1}(0)})$, from (\ref{EquivDistances.1}) one has, with $\sigma$ defined in definition \ref{WeightedSob1}:
\begin{align*}
\sup_{y\in \mathbb{R}^n\backslash\overline{B_{R_2}(0)}}|f\circ\Phi^{-1}(y)|\sigma(y)^{-\delta} \lesssim \sup_{x\in \mathbb{R}^n\backslash\overline{B_{R_1}(0)}}|f(x)|\sigma(x)^{-\delta},
\end{align*}
which proves (\ref{A-HolderProp.1}).

Concerning the second item, first notice that if $\Phi-\mathrm{Id}\in C^{m}_{1-\kappa}(\mathbb{R}^n\backslash\overline{B_{R_1}(0)})$, denoting by $u\doteq \mathrm{Id}-\Phi$, then
\begin{align*}
\Phi^{-1}=\mathrm{Id} + u\circ\Phi^{-1}=\mathrm{Id} + Au.
\end{align*}
Since from (\ref{A-HolderProp.1}) we know that $Au\in C^{0}_{1-\kappa}(\mathbb{R}^n\backslash\overline{B_{R_2}(0)})$, then we have already obtained that
\begin{align}\label{Invariance-C0}
\Phi^{-1}-\mathrm{Id} \in  C^{0}_{1-\kappa}(\mathbb{R}^n\backslash\overline{B_{R_2}(0)}).
\end{align}

Now, if $m\geq 1$, then $d\Phi-\mathrm{Id}\in C^{m-1}_{-\kappa}(\mathbb{R}^n\backslash\overline{B_{R_1}(0)})$ so, given any $\epsilon>0$, there is some $R_{\epsilon}\geq R_1$ such that 
\begin{align*}
\sup_{x\in \mathbb{R}^n\backslash\overline{B_{R_{\epsilon}}}(0)}|d\Phi-\mathrm{Id}|(x)<\epsilon.
\end{align*}
This, in particular, implies that $d\Phi^{-1}$ can be computed in terms of a Neumann series:
\begin{align*}
d\Phi^{-1}=\mathrm{Id} + \sum_{s=1}^{\infty}\varphi^s,
\end{align*}
where $\varphi\doteq \mathrm{Id}-d\Phi\in C^{m-1}_{-{\kappa}}(\mathbb{R}^n\backslash\overline{B_{R_{\epsilon}}}(0))$.  One can quickly show, as in the multiplication rule of proposition \ref{Wkpdeltaalgebra} that $\varphi^s \in C^{m-1}_{-s{\kappa}}(\mathbb{R}^n\backslash\overline{B_{R_{\epsilon}}}(0))$, and in particular that the sequence $\varphi_m\doteq \sum_{s=1}^{m}\varphi^s$ converges in $C^{m-1}_{-{\kappa}}(\mathbb{R}^n\backslash\overline{B_{R_{\epsilon}}}(0))$ to a function $\phi$. That is,
\begin{align}\label{Invariance-TopOrderCritical}
\phi=\sum_{s=1}^{\infty}\varphi^s\in C^{m-1}_{-\kappa}(\mathbb{R}^n\backslash\overline{B_{R_{\epsilon}}}(0)),
\end{align}
and therefore
\begin{align}\label{Invariance-Recursive}
d\Phi^{-1}=\mathrm{Id} + \phi\circ\Phi^{-1}.
\end{align}
But then (\ref{A-HolderProp.1}) again implies $A\phi\in C^{0}_{-\kappa}(\mathbb{R}^n\backslash\overline{B_{R_{2}}}(0))$, so that (\ref{Invariance-Recursive}) implies $\Phi^{-1}-\mathrm{Id}\in C^{1}_{1-\kappa}(\mathbb{R}^n\backslash\overline{B_{R_{2}}}(0))$.

To establish the higher order versions of the above, let us consider first the following more general situation. Let $f\in C^{k}_{\delta}(\mathbb{R}^n\backslash\overline{B_{R_{1}}}(0))$, $\Phi-\mathrm{Id}\in C^{k}_{1-\kappa}(\mathbb{R}^n\backslash\overline{B_{R_{1}}}(0))$ and $\Phi^{-1}-\mathrm{Id}\in C^{k}_{1-\kappa}(\mathbb{R}^n\backslash\overline{B_{R_{1}}}(0))$ for some $k\geq 1$ and any $\delta\in \mathbb{R}$. Let $\alpha=(\alpha_1,\cdots,\alpha_N)$ be a given multi-index, $\alpha_i\in \mathbb{N}_0$, with $k\geq |\alpha|=\sum_{i=1}^{N}\alpha_i\geq 1$, so that we can more explicitly use the Fa\`a di Bruno formula (see \cite[Theorem 11.54]{Leoni}):
\begin{align*}
\partial_{y^{\alpha}}^{|\alpha|}(Af)(y)=\sum^{{|\alpha|}}_{|\beta|=1}C_{\alpha,\beta,\gamma,l}A(\partial_{x^{\beta}}^{|\beta|}f)(y)\prod_{i=1}^{|\beta|}\partial_{y^{\gamma_i}}^{|\gamma_i|}(\Phi^{-1})^{l_i}(y),
\end{align*}
where above the of multi-indices $\gamma_i$, $i=1,\cdots,|\beta|$, must satisfy $|\gamma_i|\geq 1$ and $\sum_{i=1}^{|\beta|}\gamma_i=\alpha$, while $l=(l_1,\cdots,l_{|\beta|})$ with $1\leq l_i\leq n$, $i=1,\cdots,|\beta|$. Then, one sees that
\begin{align*}
\sup_{y\in \mathbb{R}^n\backslash\overline{B_{R_2}(0)}}|\partial^{\alpha}(Af)(y)|\sigma(y)^{-\delta-|\alpha|}&\lesssim \sum^{|\alpha|}_{|\beta|=1}\sup_{y\in \mathbb{R}^n\backslash\overline{B_{R_2}(0)}}|(\partial^{\beta}f)(\Phi^{-1}(y))|\underbrace{\prod_{i=1}^{|\beta|}|\partial^{|\gamma_i|}(\Phi^{-1})^{l_i}(y)|}_{\leq \Vert d\Phi^{-1}\Vert^{|\beta|}_{C^{|\alpha|-1}(\mathbb{R}^n\backslash\overline{B_{R_2}(0)})}}\sigma(y)^{-\delta-|\alpha|},\\
&{\lesssim \Vert d\Phi^{-1}\Vert^{|\alpha|}_{C^{|\alpha|-1}(\mathbb{R}^n\backslash\overline{B_{R_2}(0)})}\sum^{|\alpha|}_{|\beta|=1}\sup_{y\in \mathbb{R}^n\backslash\overline{B_{R_2}(0)}}|(\partial^{\beta}f)(\Phi^{-1}(y))|\sigma(y)^{-\delta-|\alpha|}},\\
&\lesssim \Vert d\Phi^{-1}\Vert^{|\alpha|}_{C^{|\alpha|-1}(\mathbb{R}^n\backslash\overline{B_{R_2}(0)})}\sum^{|\alpha|}_{|\beta|=1}\sup_{x\in \mathbb{R}^n\backslash\overline{B_{R_1}(0)}}|(\partial^{\beta}f)(x)|\sigma(x)^{-\delta-|\alpha|}.
\end{align*}
Putting the above inequality for any $1\leq |\alpha|\leq k$ with the zero order case, we find that there is a fixed constant $C>0$, naturally depending on the $C^{k-1}$-norm of $d\Phi^{-1}$, such that for all $ f\in C^{m}_{\delta}(\mathbb{R}^n\backslash\overline{B_{R_1}(0)})$:
\begin{align}\label{A-continuity}
\Vert Af\Vert_{C^{k}_{\delta}(\mathbb{R}^n\backslash\overline{B_{R_2}(0)})}\leq C \Vert f\Vert_{C^{k}_{\delta}(\mathbb{R}^n\backslash\overline{B_{R_1}(0)})}.
\end{align}

We apply the above result, in particular, to $\phi$ in (\ref{Invariance-Recursive}) to obtain the following chain of implications:
\begin{align}
\begin{split}
\text{If } \phi\in C^{k}_{-\kappa}(\mathbb{R}^n\backslash\overline{B_{R_1}(0)}) \text{ and } \Phi^{-1}-\mathrm{Id}\in C^{k}_{1-\kappa}(\mathbb{R}^n\backslash\overline{B_{R_2}(0)})&\Longrightarrow A\phi\in C^{k}_{-\kappa}(\mathbb{R}^n\backslash\overline{B_{R_2}(0)}),\\
&\Longrightarrow \Phi^{-1}-\mathrm{Id}\in C^{k+1}_{1-\kappa}(\mathbb{R}^n\backslash\overline{B_{R_2}(0)})
\end{split}
\end{align}
Since the analysis below (\ref{Invariance-Recursive}) establishes the premises of the above implication for $k=0$, we may then iterate until $k=m-1$, since from (\ref{Invariance-TopOrderCritical}) $\phi\in C^{m-1}_{-\kappa}(\mathbb{R}^n\backslash\overline{B_{R_1}(0)})$ is the best possible a priori control for $\phi$. Thus, we find
\begin{align}
\Phi^{-1}-\mathrm{Id}\in C^{m}_{1-\kappa}(\mathbb{R}^n\backslash\overline{B_{R_2}(0)}).
\end{align}
Putting this together with (\ref{A-continuity}), we find (\ref{A-HolderProp.2}).
\end{proof}

\begin{cor}\label{SobolevInvariance-Coro1}
Consider the setting described above for a coordinate change. Given $\kappa>0$, $m\geq 1$ and $\delta\in \mathbb{R}$, $\Phi-\mathrm{Id}\in C_{1-\kappa}^{m}(\mathbb{R}^n\backslash \overline{B_{R_1}(0)})$, then $\Phi^{-1}-\mathrm{Id}\in C_{1-\kappa}^{m}(\mathbb{R}^n\backslash \overline{B_{R_2}(0)})$ and moreover, for any $1<p<\infty$, the operator
\begin{align}\label{A-SobolevProp1}
A:W_{\delta}^{m,p}(\mathbb{R}^n\backslash \overline{B_{R_1}(0)})\to W_{\delta}^{m,p}(\mathbb{R}^n\backslash \overline{B_{R_2}(0)}) \text{ is bounded }.
\end{align}
\end{cor}
\begin{proof}
The implication $\Phi^{-1}-\mathrm{Id}\in C_{1-\kappa}^{m}(\mathbb{R}^n\backslash \overline{B_{R_2}(0)})$ follows from Lemma \ref{ChangeOfCoordinatesHolderLemma}, and then (\ref{A-SobolevProp1}) follows from Lemma \ref{ChangeOfCoordinatesLemma}.
\end{proof}

\begin{cor}\label{SobolevInvariance-Coro2}
Consider the setting described above for a coordinate change. Given $\kappa>0$, $m\geq 1$ and $\delta\in \mathbb{R}$, $\Phi-\mathrm{Id}\in W_{1-\kappa}^{m+1,q}(\mathbb{R}^n\backslash \overline{B_{R_1}(0)})$ for some $q>n$, then $\Phi^{-1}-\mathrm{Id}\in W_{1-\kappa}^{m+1,q}(\mathbb{R}^n\backslash \overline{B_{R_2}(0)})$ and moreover, for any $1<p<\infty$, the operator
\begin{align}\label{A-SobolevProp2}
A:W_{\delta}^{m,p}(\mathbb{R}^n\backslash \overline{B_{R_1}(0)})\to W_{\delta}^{m,p}(\mathbb{R}^n\backslash \overline{B_{R_2}(0)}) \text{ is bounded }.
\end{align}
\end{cor}
\begin{proof}
Since $q>n$, then $W_{1-{\kappa}}^{m+1,q}(\mathbb{R}^n\backslash \overline{B_{R_1}(0)})\hookrightarrow C^{m}_{1-{\kappa}}(\mathbb{R}^n\backslash \overline{B_{R_1}(0)})$, and thus Corollary \ref{SobolevInvariance-Coro1} implies both $\Phi^{-1}-\mathrm{Id}\in C_{1-{\kappa}}^{m}(\mathbb{R}^n\backslash \overline{B_{R_2}(0)})$ and (\ref{A-SobolevProp2}). The extra Sobolev regularity for $\Phi^{-1}-\mathrm{Id}$ in this case comes from the fact that the sequence $\varphi_m$ in (\ref{Invariance-TopOrderCritical}) is now Cauchy in $W^{m,q}_{-{\kappa}}(\mathbb{R}^n\backslash\overline{B}_{R_{\epsilon}(0)})$, $R_{\epsilon}$ sufficiently large. This is a consequence of the fact that $W^{m,q}_{-{\kappa}}(\mathbb{R}^n\backslash\overline{B_{R_1(0)}})$ is an algebra under multiplication due to Lemma \ref{MultiplicationPropertyGeneral} and $q>n$. Therefore $\phi\in W^{m,q}_{-{\kappa}}(\mathbb{R}^n\backslash\overline{B}_{R_{\epsilon}(0)})$ and
\begin{align*}
d\Phi^{-1}=\mathrm{Id} + \phi\circ\Phi^{-1}=\mathrm{Id} + A\phi
\end{align*}
Then, (\ref{A-SobolevProp2}) implies $A\phi\in W^{m,q}_{-{\kappa}}(\mathbb{R}^n\backslash\overline{B_{R_y}(0)})$. That is, $d\Phi^{-1}-\mathrm{Id} \in W^{m,q}_{-{\kappa}}(\mathbb{R}^n\backslash\overline{B_{R_y}(0)})$. Since $\Phi^{-1}-\mathrm{Id} \in L^{q}_{1-{\kappa}}(\mathbb{R}^n\backslash\overline{B}_{R_2}(0))$ is given from Lemma \ref{ChangeOfCoordinatesLemma}, then $\Phi^{-1}-\mathrm{Id} \in W^{m+1,q}_{1-{\kappa}}(\mathbb{R}^n\backslash\overline{B}_{R_2}(0))$, which finishes the proof. 
\end{proof}

{The above results, in particular Corollary \ref{SobolevInvariance-Coro2}, will be important in establishing the following theorem about the existence of harmonic coordinate at infinity, specially to be able to prove that controls in a given a priori structure of infinity are preserved when changing coordinate to a harmonic structure of infinity.}

\begin{theo}\label{HarmonicCoordThm}
Let $(M,g)$ be a smooth AE manifold, with $g$ a $W^{k,q}_{-\tau}$-AE metric, $k\geq 1$, $q>n$, with respect to a structure of infinity $\Phi_x :\, M\backslash K\rightarrow  \mathbb{R}^n\backslash\overline{B_{R_x}(0)}$ where $K\subset\subset M$, with $1-\tau\not\in \mathbb{Z}\backslash\{-1,-2,\cdots,3-n\}$,\footnote{The integer values highlighted in the set $\mathbb{Z}\backslash\{-1,-2,\cdots,3-n\}$ are referred to as \emph{exceptional}, and are particular weights for which $\Delta_g$ loses Fredholm properties. For further details we refer the reader to \cite{MR849427}.} 
 and fix $1<\eta<2$. There are functions $y^{i}\in W^{k+1,q}_{\eta}$, $i=1,\cdots,n$, such that $\Delta_gy^i=0$ and $(x^i-y^i)\in W^{k+1,q}_{1-\tau^{*}}(\mathbb{R}^n\backslash\overline{{B_{R_x}(0)})}$ for $\tau^{*}\doteq \min\{\tau,n-2\}$.
As a result there exists a structure of infinity given by these harmonic coordinates $\Phi_y : \, M \backslash K' \rightarrow \mathbb{R}^n \backslash \overline{B_{R_y}(0)}$  where $K'\subset\subset M$ in which $(M,g)$ is $W^{k,q}_{-\tau}$-AE  and such that any tensor $T \in W^{l,p}_{\delta}( M, \Phi_x)$, $0 \le l \le k$, ${1<p\leq q}$, {$\delta\in \mathbb{R}$}  also satisfies $T \in W^{l,p}_{\delta}(M, \Phi_y)$. 
\end{theo}
\begin{proof}
Let us first extend the functions $x^i$ smoothly to all of $M$. Then, near infinity, $\Delta_gx^i=g^{kl}\Gamma^i_{kl}\doteq \Gamma^i\in W^{k-1,q}_{-1-\tau}(M, \Phi_x)$. Notice that if $1-\tau>2-n$, then Proposition \ref{regularitymaxwell} implies the existence of $v^i\in W^{2,q}_{1-\tau}(M, \Phi_x)$ solving $\Delta_gv^i=\Gamma^i$. Now, if $1-\tau< 2-n$ (the equality case is an exceptional case), we cannot, a priori, improve the decay given by $|x|^{2-n}$. In any case, we know that there are solutions $v^i\in W^{2,q}_{1-\tau^{*}}(M, \Phi_x)$ to:
\begin{align}\label{BootstrapEq.1}
\Delta_gv^i=\Gamma^i,
\end{align}
where $\tau^{*}=\min\{\tau,n-2\}$ and $\Gamma^i\in W^{k-1,q}_{-1-\tau}(M, \Phi_x)$. Therefore, there is some $v^i\in W^{2,q}_{1-\tau^{*}}(M, \Phi_x)$ satisfying $\Delta_g(x^i-v^i)=0$. Defining $y^i\doteq x^i-v^i$, Lemma \ref{lemregularite} implies $y^i-x^i\in W^{k+1,q}_{1-\tau^{*}}(M, \Phi_x)$. Let us define  $\Phi_y = (y^1, \dots, y^n ) : \, M \backslash K \rightarrow E_y\doteq \R^n\backslash\overline{B_{R_y}(0)}$, and $y(x)=\Phi_y \circ \Phi_x^{-1}$, $\nu^i(x) =v^i \circ \Phi_x^{-1} $ and $\nu = (\nu^1, \dots , \nu^n)$.

By definition $y^i= x^i-\nu^i$, and since $\nabla \nu\in W^{k,q}_{-\tau^*}(\R^n \backslash \overline{ B_{R_x}(0)})$, by Sobolev embeddings $\frac{\partial y^i}{\partial x^j}=\delta^i_j + o(|x|^{-\tau^{*}})$. Thus $y(x)$ is a diffeomorphism for $|x|$ sufficiently large. That is, there is some $R_x>0$ such that $\Phi_y\circ\Phi^{-1}_x: E_x\doteq \mathbb{R}^n\backslash\overline{B_{R_x}(0)}\to \Phi_y\circ\Phi^{-1}_x(E_x)\subset \mathbb{R}^n$ is a diffeomorphism. Notice also that $y-x\in W^{k+1,q}_{1-\tau^{*}}(\Phi_x)\hookrightarrow C^{k-1}_{1-\tau^{*}}(E_x)$, $k-1\geq 0$, and therefore there is a constant $C>0$ such that
\begin{align}\label{EquivDistances}
C^{-1}|x|\leq |y(x)|=|\Phi_y\circ \Phi^{-1}_{x}(x)|\leq C|x|.
\end{align}
This implies that the diffeomorphism $\Phi_y\circ \Phi^{-1}_{x}$ maps a subset of $\mathbb{R}^n\backslash\overline{B_{R_x}(0)}$ onto $\mathbb{R}^n\backslash\overline{B_{R_y}(0)}$ for some $R_y>0$. Hence, there is some compact set $K_y\subset M$, such that $\Phi_y:E_y\doteq M\backslash K_y\to \mathbb{R}^n\backslash \overline{B_{R_y}(0)}$, and therefore $\Phi_y$ provides a \emph{harmonic structure of infinity for $M$}. Furthermore, since the diffeomorphism $\Phi_y\circ\Phi^{-1}_x-\mathrm{Id}=\nu\in W^{k+1,q}_{1-\tau^{*}}(\mathbb{R}^n\backslash\overline{B_{R_x}(0)})$, then Corollary \ref{SobolevInvariance-Coro2} shows that $\Phi_x\circ\Phi^{-1}_y-\mathrm{Id}\in W^{k+1,q}_{1-\tau^{*}}(\mathbb{R}^n\backslash\overline{B_{R_y}(0)})$.

Finally, to establish that for any tensor field $T \in W^{l,p}_{\delta}( M, \Phi_x)$, $0 \le l \le k$, ${1<p\leq q}$ it holds that $T \in W^{l,p}_{\delta}(M, \Phi_y)$, we need only apply the coordinate transformation rule, together with the observations (which all follow from Corollary \ref{SobolevInvariance-Coro2}) $d(\Phi_x\circ\Phi^{-1}_y)-\mathrm{Id}\in W^{k,q}_{-\tau^{*}}(\mathbb{R}^n\backslash\overline{B_{R_y}(0)})$, $(d(\Phi_y\circ\Phi^{-1}_x))\circ\Phi_x\circ\Phi^{-1}_y -\mathrm{Id}\in W^{k,q}_{-\tau^{*}}(\mathbb{R}^n\backslash\overline{B_{R_y}(0)})$ and $(T\circ\Phi^{-1}_x)\circ(\Phi_x\circ\Phi^{-1}_y)\in W^{l,{p}}_{\delta}(\mathbb{R}^n\backslash\overline{B_{R_y}(0)})$. Then, since $q>n$ and {$p\leq q$}, then ${W^{k,q}_{-\tau^{*}}\otimes W^{l,{p}}_{\delta}\hookrightarrow W^{l,{p}}_{\delta}}$ for all $0\leq l\leq k$ and all $\delta\in \mathbb{R}$, both properties due to Lemma \ref{MultiplicationPropertyGeneral}. Thus, for example, if $T$ is a $(0,2)$-tensor field, then applying the coordinate transformation rule to any tensor field $T\in W^{l,{p}}_{\delta}(M,\Phi_x)$ we find
\begin{align*}
T(\partial_{y^i},\partial_{y^j})&=\partial_{y^i}(\Phi_x\circ\Phi^{-1}_y)^a\partial_{y^j}(\Phi_x\circ\Phi^{-1}_y)^bT(\partial_{x^a},\partial_{x^b})\circ(\Phi_x\circ\Phi^{-1}_y).
\end{align*}
From the above discussion we know that $\partial_{y^i}(\Phi_x\circ\Phi^{-1}_y)^a-\mathrm{Id}\in W^{k,{p}}_{-\tau^{*}}(\mathbb{R}^n\backslash\overline{B_{R_y}(0)})$, while the composition $T(\partial_{x^a},\partial_{x^b})\circ(\Phi_x\circ\Phi^{-1}_y)\in W^{l,{p}}_{\delta}(\mathbb{R}^n\backslash\overline{B_{R_y}(0)})$ from Corollary \ref{SobolevInvariance-Coro2}. Then, since $q>n$, Lemma \ref{MultiplicationPropertyGeneral} implies $T(\partial_{y^i},\partial_{y^j})\in W^{l,{p}}_{\delta}(\mathbb{R}^n\backslash\overline{B_{R_y}(0)})$.
\end{proof}

\section{ $\mathcal{E}(g)$ as a function of the $J$-Einstein tensor, proof of Theorem \ref{HerzlichThmIntro}}


Let us start by introducing a few preliminary results and establishing some useful notations. First, using the same notations as in Definition \ref{AEManifolds2}, let us introduce a smooth metric $e$ on our manifold $M$ such that, on each end $N_i=\Phi_i(\mathbb{R}^n\backslash\overline{B_1(0)})$, $e_{kl}={\Phi_i^{-1}}^{*}\delta_{kl}$, where $\delta$ stands for the Euclidean metric on $\mathbb{R}^n\backslash\overline{B_1(0)}$. Consider $n\geq 3$, $k\geq 4$ and $k>\frac{n}{p}$ and define the set
\begin{align}
\begin{split}
\mathcal{M}^{k,p}_{-\tau}(M)\doteq \{ g\in W^{k,p}_{loc}(M;S_2M) \: : \: g-e\in W^{k,p}_{-\tau}(M;S_2M) \text{ and } g>0 \}.
\end{split}
\end{align}

We look at $\mathcal{M}^{k,p}_{-\tau}$ as an affine space where $g\in\mathcal{M}^{k,p}_{-\tau}$ is given by $g=e+h$, with $h\in W^{k,p}_{-\tau}$. Thus, the map
\begin{align*}
\mathcal{M}^{k,p}_{-\tau}&\to W^{k,p}_{-\tau},\\
g&\mapsto h=g-e
\end{align*}
is a one to one linear map, and we identify $\mathcal{M}^{k,p}_{-\tau}$ with its image under this map in $W^{k,p}_{-\tau}$. Notice that these \emph{restrictions} on the allowable metrics are only restrictions at infinity, which, just as in Definition \ref{AEmanifolds}, concern the behaviour of $g$ near infinity in our chosen end coordinates. Let us now recall some formulae for geometric objects, in particular for the Ricci tensor which can be written as (see \cite{MR2473363}):
\begin{align}\label{gauge0}
\begin{split}
R_{ij}(g)&=-\frac{1}{2}g^{ab}D_{a}D_{b}g_{ij}+f_{ij}(g,Dg)+\frac{1}{2}(g_{ia}D_{j}F^{a}+g_{ja}D_{i}F^{a}),\\
F^{a}&\doteq g^{kl}(\Gamma^{a}_{kl}(g)-\Gamma^{a}_{kl}(e))\doteq g^{kl}S^{a}_{kl},\;\; ; \;\; S^{a}_{kl}\doteq \frac{g^{ab}}{2}(D_{k}g_{bl} + D_{l}g_{bk} - D_{b}g_{kl}).
\end{split}
\end{align}
where $D$ stands for the $e$-covariant derivative, $\Gamma$ for the corresponding Christoffel symbols and
\begin{align}\label{f-tensor}
\begin{split}
\!\!f_{ij}&\doteq   - \frac{1}{2}\big\{  D_{i}g^{ab}D_{a}g_{b j}  + D_{j}g^{ab}D_{a}g_{b i} \}  + \frac{1}{2}F^{a}D_{a}g_{ij} - S^{a}_{b i}S^{b}_{a j} - \frac{1}{2}g^{kl} \big\{ R^{a}_{kl i}(e)g_{a j}  + R^{a}_{kl j}(e) g_{a i} \big\}.
\end{split}
\end{align}
Above, $R^{i}_{jkl}\doteq dx^{i}(R(\partial_k,\partial_l)\partial_j)$, where $R(X,Y)Z=\nabla_{X}\nabla_{Y}Z - \nabla_{Y}\nabla_{X}Z - \nabla_{[X,Y]}Z$, $X,Y,Z\in\Gamma(TM)$, denotes the curvature tensor, establishing our curvature conventions. Then, we can see $\mathrm{Ric}$ as a map on $\mathcal{M}^{k,p}_{-\tau}$, $g\cong h\mapsto \mathrm{Ric}_g$, given by
\begin{align*}
\mathrm{Ric}_g&=-\frac{1}{2}{e^{-1}}_{\cdot}D^2 h - \frac{1}{2}(g^{-1} - e^{-1})_{\cdot}D^2 h + f(h,Dh)+\frac{1}{2}(e_{\cdot}DF + e_{\cdot}DF) + \frac{1}{2}(h_{\cdot}DF + h_{\cdot}DF),
\end{align*}
where, given two tensors $T$ and $T'$, we have used the notation $T_{\cdot}T'$ to denote an arbitrary contraction, which are understood from (\ref{gauge0}) in our case. Notice that, using similar manipulations to those presented above, the tensor fields $S$ and $F$ appearing in (\ref{gauge0}) can be seen as maps on $\mathcal{M}^{k,p}_{-\tau}$, $h\mapsto S_g,F_g$.

\begin{lem}\label{QcurvFrechReg}
Let $(M^n,g)$ be a $W^{k,p}_{-\tau}$-AE manifold with $n\geq 3$, $p>1$, $k\geq4$ and $k>\frac{n}{p}+2$. Then, the $Q$-curvature defines a smooth non-linear map 
\begin{align}
\begin{split}
Q:\mathcal{M}^{k,p}_{-\tau}(M)\to W^{k-4,p}_{-\tau-4}(M). 
\end{split}
\end{align}
\end{lem}
\begin{proof}
Appealing to (\ref{gauge0})-(\ref{f-tensor}) and Definition \ref{AEmanifolds}, since $D g\in W^{k-1,p}_{-\tau-1}$ and Lemma \ref{MultiplicationPropertyGeneral} shows that $S\in W^{k-1,p}_{-\tau-1}$, we find  that $F	\in W^{k-1,p}_{-\tau-1}$. Then, since $W^{k-1,p}_{-\tau-1}\otimes W^{k-1,p}_{-\tau-1}\hookrightarrow W^{k-1,p}_{-\tau-2}$ due to Lemma \ref{MultiplicationPropertyGeneral}, we find $\mathrm{Ric}_g\in W^{k-2,p}_{-\tau-2}$. This, in turn, directly implies that $R_g\in W^{k-2,p}_{-\tau-2}$ and that $|\mathrm{Ric}_g|^2_g,R^2_g\in W^{k-2,p}_{-\tau-4}$. Finally, since $\Delta_gR_g\in W^{k-4,p}_{-\tau-4}$, we see that $Q_g\in W^{k-4,p}_{-\tau-4}$. 

Concerning the Frechét regularity of the map $Q:\mathcal{M}^{k,p}_{-\tau}\to W^{k-4,p}_{-\tau-4}(M)$, let us appeal to some basic results, such as those established in \cite[Chapter 2]{AMR-Book}. First, notice that since bounded linear maps are smooth maps between normed spaces, we know that $h\mapsto D^{l}g=D^lh$ is a smooth map from $W^{k,p}_{-\tau}\to W^{k-l,p}_{-\tau-l}$, $0\leq l\leq k$. Also, concerning non-linear maps, our regularity $k>\frac{n}{p}+2$ and $\tau>0$ allow us to use classic pointwise expressions and multiplication properties of Sobolev spaces to guarantee that $g\mapsto g^{-1}$ is a smooth map from $\mathcal{M}^{k,p}_{-\tau}\to {\mathcal{M}^{*}}^{k,p}_{-\tau}$.\footnote{See, for instance, \cite[Chapter III]{GirbauBruna} for the detailed computations.}${}^{,}$\footnote{We are using the notation ${\mathcal{M}^{*}}^{k,p}_{-\tau}$ to define the analogous space of \emph{metrics} pointwise defined on cotangent spaces.} Also, notice that given two tensor fields $T_1$ and $T_2$, a tensor contraction defines a bilinear map $(T_1,T_2)\mapsto {T_1}_{\cdot}T_2$. From Lemma \ref{MultiplicationPropertyGeneral}, such bilinear maps are continuous maps from $W^{k_1,p}_{-\tau_1}\times W^{k_2,p}_{-\tau_2}\to W^{k_3,p}_{-\tau_3}$ as long as $k_1+k_2>\frac{n}{p}+k_3$ and $\tau_3<\tau_1+\tau_2$. Since bounded bilinear maps are also smooth maps between normed spaces, putting these observations together we see that $F,S,f:\mathcal{M}^{k,p}_{-\tau}\to W^{k-2,p}_{-\tau-2}$ are smooth maps under our hypotheses, implying that $\mathrm{Ric}:\mathcal{M}^{k,p}_{-\tau}\to W^{k-2,p}_{-\tau-2}$ is smooth, and the same holds for the scalar curvature map $g\mapsto R_g$. Along these lines, the quadratic terms $|\mathrm{Ric}_g|^2_g$ and $R_g^2$ are both smooth maps from $\mathcal{M}^{k,p}_{-\tau}\to W^{k-4,p}_{-\tau-4}$, and we can rewrite
\begin{align}
\Delta_gR_g
&=\Delta_eR_g + (g^{-1}-e^{-1})_{\cdot}D^2R_g - {e^{-1}}_{\cdot}S_{\cdot}DR_g - (g^{-1}-e^{-1})_{\cdot}S_{\cdot}DR_g,
\end{align}
which, from the above considerations, is a smooth map from $W^{k-2,p}_{-\tau-2}\to W^{k-4,p}_{-\tau-4}$. Putting all this together establishes the lemma.
\end{proof}

Now, given an annulus near infinity, we intend to expand $Q_g$ around the Euclidean metric. Thus, consider annuli $\Omega_k=\{x\in\mathbb{R}^n \: :\: R_0\leq |x|\leq R_k\}\subset E_i$ within the end of $M$, and  $\hat{g}=\chi g + (1-\chi)\delta$ a metric on $\mathbb{R}^n$, where $\chi$ is a cut-off function, $0\leq\chi\leq 1$, equal to one on a neighbourhood of infinity and zero inside some ball, and satisfying $|\nabla^k \chi|\leq C|x|^{-k}$, so that $\hat{g}$ is Euclidean within some ball in $\mathbb{R}^n$ and agrees with $g$ in a neighbourhood of infinity. Since $g_{ij}=\delta_{ij}+O(|x|^{-\tau})$, then 
\begin{align}\label{ghatAE}
\hat{g}&=\chi \delta_{ij} + O_4(|x|^{-\tau}) + (1-\chi)\delta_{ij}=\delta_{ij}+{O(|x|^{-\tau}).}
\end{align}
If we fix $R_0$ to be sufficiently large, then $h \doteq\hat{g}-\delta$ is small on $\Omega_k$ and, due to Lemma \ref{QcurvFrechReg} and \cite[Theorem 2.4.15]{AMR-Book}, we can make a Taylor expansion of $Q:\mathcal{U}_{\delta}\subset\mathcal{M}^{k,p}_{-\tau}(\mathbb{R}^{n}\backslash\overline{B_{R_0}(0)})\to W^{k-4,p}_{-\tau-4}(\mathbb{R}^{n}\backslash\overline{B_{R_0}(0)})$, where $\mathcal{U}_{\delta}$ denotes a sufficiently small neighbourhood of $\delta$,\footnote{Due to our identifications of $\mathcal{M}^{k,p}_{-\tau}$ with a subset of $W^{k,p}_{-\tau}$ via $g=e+h\mapsto h$, $\mathcal{U}_{\delta}$ is actually identified with a neighbourhood of the origin in $W^{k,p}_{-\tau}$.} which gives us
\begin{align}
Q_{\hat{g}}=Q_{\delta}+DQ_{\delta}\cdot h + \mathcal{R}(h)=DQ_{\delta}\cdot h + \mathcal{R}(h),
\end{align} 
where $\mathcal{R}(h)$ stands for the remainder. Letting $V\in C^{\infty}(\mathbb{R}^{n}\backslash\overline{B_{R_0}(0)})$, one finds
\begin{align}\label{Michell1}
\int_{\Omega_k}VQ_{\hat{g}}dV_{\delta}&=\int_{\Omega_k}VDQ_{\delta}\cdot h dV_{\delta} + \int_{\Omega_k}V\mathcal{R}(h) dV_{\delta}.
\end{align}
Direct computations, which we will omit here\footnote{In this case, one can see that since the curvatures of the Euclidean space are null, the linearization of the quadratic terms in the $Q$-curvature, as well as the linearization of the Laplacian at the Euclidean metric will be zero. This ensures that $DQ_{\delta}\cdot h=-\frac{1}{2(n-1)}  \Delta_\delta D\mathrm{R}_{\delta}\cdot h$ with      \cite[Chapter I (11.5)]{MR2473363} yielding the linearization of  the scalar curvature, and the result.}, yield
\begin{align*}
DQ_{\delta}\cdot h=-\frac{1}{2(n-1)}\left(\Delta_{\delta}(\mathrm{div}_{\delta}^2 h) - \Delta^2_{\delta}(\mathrm{tr}_{\delta}h) \right),
\end{align*}
 where $\mathrm{div}_{\delta}^2 h \doteq  \partial_{ij} h_{ij}$. We can then integrate the first term in the right-hand side of (\ref{Michell1}) by parts to obtain
\begin{align}\label{Michell2}
\int_{\Omega_k}VQ_{\hat{g}}dV_{\delta}&=\int_{\Omega_k}\langle DQ^{*}_{\delta}\cdot V, h\rangle_{\delta} dV_{\delta} + \int_{\partial\Omega_k}\langle \mathbb{U}(h,V),\nu_{\delta}\rangle_{\delta}d\omega_{\delta} + \int_{\Omega_k}V\mathcal{R}(h) dV_{\delta},
\end{align}
where $\nu_\delta$ denotes the exterior normal. We can explicitly compute the $1$-form $\mathbb{U}$ in the boundary term and the adjoint $DQ^{*}_{\delta}$:
\begin{align}\label{Michell3}
\begin{split}
\mathbb{U}(h,V)&=-\frac{1}{2(n-1)}(Vd u - udV + \Delta_{\delta}V(\mathrm{div}_{\delta}h-d\mathrm{tr}_{\delta}h)-h(d\Delta_{\delta}V,\cdot)+\mathrm{tr}_{\delta}hd\Delta_{\delta}V),\\
D^{*}Q_{\delta}\cdot V&= -\frac{1}{2(n-1)}\left(-\Delta^2_{\delta}V\: \delta + \partial^2\Delta_{\delta}V\right)
\end{split}
\end{align}
where we denoted
\begin{align}
u\doteq \mathrm{div}_{\delta}^2h - \Delta_{\delta}\mathrm{tr}_{\delta}h.
\end{align}

Now, along the lines of \cite{Michel,Herzlich}, given a Riemannian manifold $(M^n,g)$, let us distinguish the following set:
\begin{align}
\mathcal{N}_{0}\doteq \mathrm{Ker}(D^{*}Q_{\delta}).
\end{align}
Tracing \eqref{Michell3} ensures that functions in $\mathcal{N}_0$ are biharmonic, and thus that they are smooth.

In this context, we can now establish the following result.

\begin{lem}\label{MichellLemma1}
Consider a {smooth} AE manifold $(M^n,g)$ {satisfying the hypotheses of Lemma \ref{QcurvFrechReg}} and associate to it the AE manifold $(\mathbb{R}^n,\hat{g})$ with $\hat{g}$ defined as above. Given $V\in\mathcal{N}_0$, if $VQ_{g},V\mathcal{R}(h)\in L^1(\mathbb{R}^n\backslash\overline{B_1(0)},dV_{\hat{g}})$ then the asymptotic charge
\begin{align}\label{4thEnergyHerz1}
\mathbb{E}_{g}(V)\doteq\lim_{r\rightarrow \infty}\int_{S_{r}}\langle \mathbb{U}(h,V),\nu_{\delta}\rangle_{\delta}d\omega_{\delta}
\end{align}
is finite and its value is independent of the sequence of compact sets $S_r$ used to compute it. 
\end{lem}
\begin{proof}
Using that $V\in\mathcal{N}_0$ in (\ref{Michell2}), we find
\begin{align*}
\int_{\Omega_k}VQ_{\hat{g}}dV_{\delta}&=\int_{\partial\Omega_k}\langle \mathbb{U}(h,V),\nu_{\delta}\rangle_{\delta}d\omega_{\delta} +\int_{\Omega_k}V\mathcal{R}(h) dV_{\delta}.
\end{align*}
Now, let $S_0$ and $S_k$ denote the \emph{inner} and \emph{outer} boundaries of $\Omega_k$, so that we can rewrite
\begin{align}\label{Michell4}
\int_{S_k}\langle \mathbb{U}(h,V),\nu_{\delta}\rangle_{\delta}d\omega_{\delta}=-\int_{S_0}\langle \mathbb{U}(h,V),\nu_{\delta}\rangle_{\delta}d\omega_{\delta} + \int_{\Omega_k}VQ_{\hat{g}}dV_{\delta} -\int_{\Omega_k}V\mathcal{R}(h) dV_{\delta}.
\end{align}
Notice that above, following the conventions of (\ref{Michell2}) where we follow the Stokes orientation for $\Omega_k$, $-\nu_{\delta}|_{S_0}$ points towards infinity. 

Due to our hypotheses, if we pass to the limit $k\rightarrow\infty$ in the above identity the last two terms must be finite, while the first one is always finite and independent of $k$. Therefore the left-hand side is finite independently of the sequence used to compute it.


\end{proof}

With the above general result in mind, let us notice that $V\equiv 1 \in \mathcal{N}_0$ and therefore the above lemma provides an implicit condition (given in the form of $\mathcal{R}(h)\in L^1(\mathbb{R}^n\backslash\overline{B_1(0)},dV_g)$) so that $\mathbb{E}_g(1)$ is finite and well-defined. In fact, in this case the 1-form $\mathbb{U}$ is reduced to
\begin{align}\label{Michell5}
\mathbb{U}(h,1)&=-\frac{1}{2(n-1)}d(\mathrm{div}_{\delta}^2 h - \Delta_{\delta}\mathrm{tr}_{\delta}h).
\end{align}

With the above expression in mind, we find the following corollary.

\begin{lem}\label{MichellLemma2}
Let $(M^n,g)$ be a {$W^{k,p}_{-\tau}$}-AE manifold, {$p>n$ and $k\geq 5$}, of order $\tau>\max\{0,\frac{n-4}{2}\}$, $n\geq 3$, satisfying $Q_g\in L^1(M,dV_g)$. Then, it holds that $\mathbb{E}_g(1)$ is well-defined and 
\begin{align}
\mathbb{E}_{g}(1)=-\frac{1}{2(n-1)}\lim_{r\rightarrow \infty}\int_{S_{r}}\langle d(\mathrm{div}_{\delta}^2 g - \Delta_{\delta}\mathrm{tr}_{\delta}g),\nu_{\delta}\rangle_{\delta}d\omega_{\delta}=\frac{1}{2(n-1)}\mathcal{E}(g),
\end{align}
where $\nu_{\delta}=\frac{x}{r}$ stands for the outward pointing Euclidean unit normal to $S_r$.
\end{lem}
\begin{proof}

Using (\ref{Michell5}) we recognise in the right-hand side of (\ref{4thEnergyHerz1}) the \emph{energy density} from the fourth order energy $\mathcal{E}(g)$ analysed in \cite{avalos2021energy} and whose definition can be found in the introduction \eqref{FourthOrderEnergy}. Thus, {since under our conditions $W^{k,p}_{-\tau}(M)\hookrightarrow C^{4}_{-\tau}(M)$}, from \cite[Proposition 1]{avalos2021positive}, it follows that we can pass to limit $r\rightarrow\infty$ under the conditions of this lemma, which ensure convergence without requiring that $\mathcal{R}(h) \in L^1$. This establishes the claim.
\end{proof}

\begin{remark}
The choice of $1 \in \mathcal{N}_0$ is actually meaningful. Let us recall that, on a relativistic initial data set $(M,g,K)$ solving the constraint equations of general relativity, a Killing initial data set (KID) is defined as a non-trivial element $(f,V)$ in the kernel of the adjoint $D\Phi_{(g,K)}^{*}$ to the linearised constraint equations, where $\Phi$ denotes the constraint map, sending $(g,K)\mapsto \Phi(g,K)=(\mathcal{H}(g,K),\mathcal{M}(g,K))$, where $\mathcal{H}$ and $\mathcal{M}$ denote the Hamiltonian and momentum constraints. KIDs are known to be in a one-to-one correspondence with Killing fields in the space-time obtained by evolving $(M,g,K)$ into a solution of the space-time Einstein equations
 \cite{MR456272,MR416469}. In particular, if $(M,g)=\left(\R^n, \delta \right)$ and $K=0$, the solution of the Cauchy problem is $\R^{n,1}=\R^n \times \R$ with the Lorentz metric of signature $(n,1)$, and the KID associated to the time translation vector $\partial_t$ is merely $(1,0)$. Generalising this case, for totally geodesic initial data sets ($K\equiv 0$), the set of KIDs reduces to the set of static potentials $\mathcal{N}^{GR}_0(g)$, which are defined as elements $f\in \mathrm{ker} DR^{*}_g$. 
Thus, $1\in \mathcal{N}^{GR}_0(\delta)$ becomes naturally associated to time-like Killing fields (and their associated conservation laws) of AE initial data sets.\footnote{For some further details on this topic, see, for instance, \cite{Herzlich}.}

In parallel to the above comments, $\mathcal{E}(g)$ was introduced in \cite{avalos2021energy} as a quantity defined on AE-slices of asymptotically Minkowskian space-time solutions to a set of fourth order equations, which is conserved through time translations. In particular, from the work of \cite{avalos2021energy} and \cite{avalos2021positive}, the $Q$-curvature of the corresponding fourth order initial data set appears as a natural analogue to the scalar curvature for the corresponding relativistic initial data sets. Thus, $\mathbb{E}_g(1)$ appears as the corresponding dynamical equivalent of the global Lagrangian approach defining $\mathcal{E}(g)$ in \cite{avalos2021energy}, with the same conserved energy on the space slices. In GR, it is shown that the same approach to the ADM mass admits a Hamiltonian formulation. It can be expected that a Hamiltonian formulation of the fourth order case would yield $\mathbb{E}_g(1)$  in the same manner.
\end{remark}

Let us recall that our fourth order $J$-Einstein tensor is defined by 
\begin{align}
\label{defGJg}
G_{J}\doteq  J_g -\frac{1}{4}Q_g\, g ,
\end{align}
where $J_g$ is defined by \eqref{J-tensor.2} and the $Q$-curvature by \eqref{Qgintro}.
From \cite{aa43,AvalosFreitas}, it is known to satisfy the local conservation identity $\mathrm{div}_gG_J=0$. Then, in parallel to \cite{Herzlich}, we have Theorem \ref{HerzlichThmIntro}, which we state again to spare the reader some back and forth.

\begin{theo}\label{HerzlichThm}
Let $(M^n,g)$ be an AE manifold satisfying the hypotheses of Lemma \ref{MichellLemma2}. Then, the following identity holds
\begin{align}
\frac{n-4}{8(n-1)} \mathcal{E}(g)=-\lim_{r\rightarrow\infty}\int_{S_r}G_J(X,\nu_{\delta})d\omega_{\delta},
\end{align}
where $X=r\partial_r$.
\end{theo}
\begin{proof}
As mentioned in the introduction, the theorem  stands for $n= 3, 4$  since for these low dimensions the fourth order mass detects non-decreasing terms at infinity. Imposing decay on the metric (and thus also on $G_{J_g}$) then forces  $\mathcal{E}(g)=0 = \lim_{r\rightarrow\infty}\int_{S_r}G_{J_g}(X,\nu_{\delta})d\omega_{\delta}$ (see \cite{avalos2021energy,avalos2021positive}).

For $n \ge 5$, let us first  choose a cut-off function $\chi_R$ satisfying
\begin{align*}
\chi_R(x)=\begin{cases}
0, \text{ if } |x|<\frac{R}{2},\\
1, \text{ if } |x|>\frac{3R}{4},
\end{cases}
\end{align*}
and such that $|\nabla^k\chi_R|\leq C_kR^{-k}$, $k\leq 4$, for some constants $C_k$ independent of $R$. Then, consider the annulus $\Omega_R=B_{R}(0)\backslash \overline{B_{\frac{R}{4}}(0)}\subset E\cong \mathbb{R}^n\backslash\overline{B_1(0)}$, for $R$ sufficiently large. Also, denote by $\hat{g}=\chi_Rg + (1-\chi_R)\delta$ an associated AE metric on $\mathbb{R}^n$, which is by construction exactly Euclidean on the inner boundary of $\Omega_R$, while it agrees with $g$ on its outer boundary.  Since $\hat{g}$ is AE,  the fourth order tensors $J_{\hat{g}}$, $Q_{\hat{g}}$, $G_{J_{\hat{g}}}$ are $O\left(|x|^{-\tau-4} \right)$.\footnote{We refer the reader to the proof of theorem \ref{theoJflatisflat} for a detailed computation of the growth order in the weak case. }

Now, from  the local conservation law obeyed by $G_{J_{\hat{g}}}$ we find that
\begin{align*}
\frac{1}{2}\int_{\Omega_R}\langle G_{J_{\hat{g}}},\pounds_X{\hat{g}}\rangle_{\hat{g}}dV_{\hat{g}}=\int_{\Omega_R}\mathrm{div}_{\hat{g}}(G_{J_{\hat{g}}}(X,\cdot))dV_{\hat{g}}=\int_{\partial\Omega_R}G_{J_{\hat{g}}}(X,\nu_{\hat{g}})d\omega_{\hat{g}}.
\end{align*}
Since $\mathrm{tr}_{\hat{g}}G_{J_{\hat{g}}}=\frac{4-n}{4}Q_{\hat{g}}$, this implies that
\begin{align}\label{Herzlich1}
\int_{\partial\Omega_R}G_{J_{\hat{g}}}(X,\nu_{\hat{g}})d\omega_{\hat{g}}&=\frac{1}{2}\int_{\Omega_R}\langle G_{J_{\hat{g}}},\pounds_{{\hat{g}},conf}X\rangle_{\hat{g}}dV_{\hat{g}} +\frac{4-n}{4n}\int_{\Omega_R}Q_{\hat{g}}\mathrm{div}_{\hat{g}}XdV_{\hat{g}},
\end{align}
where $\pounds_{{\hat{g}},conf}X\doteq \pounds_{X}{\hat{g}}-\frac{2}{n}\mathrm{div}_{\hat{g}}X\:{\hat{g}}$ stands for the conformal Killing Laplacian.  Let us now estimate each term in (\ref{Herzlich1}). First of all, notice that since ${\hat{g}}$ is Euclidean in a neighbourhood of the inner boundary of $\Omega_R$ and $g$ in a neighbourhood of the outer boundary, then 
\begin{align}\label{Herzlich2}
\int_{\partial\Omega_R}G_{J_{\hat{g}}}(X,\nu)d\omega_{\hat{g}}=\int_{S_R}G_{J_{g}}(X,\nu)d\omega_{g}.
\end{align}

Now, since $X$ is a conformal Killing field of the Euclidean metric, appealing to the  AE condition we find
\begin{align*}
\pounds_{g,conf}X=\pounds_{\delta,conf}X + O(|x|^{-\tau})=O(|x|^{-\tau}).
\end{align*}
Therefore, 
\begin{align*}
|\langle G_{J_{\hat{g}}},\pounds_{{\hat{g}},conf}X\rangle_{\hat{g}}|=O(|x|^{-2\tau-4}).
\end{align*}
This implies that
\begin{align}\label{Herzlich3}
\big\vert\int_{\Omega_R}\langle G_{J_{\hat{g}}},\pounds_{{\hat{g}},conf}X\rangle_{\hat{g}}dV_{\hat{g}}\big\vert=O(R^{n-2\tau-4})=o(1).
\end{align}

Let us now deal with the last term in (\ref{Herzlich1}), where since 
\begin{align*}
\mathrm{div}_gX=\mathrm{div}_{\delta}X+O(|x|^{-\tau})=n + O(|x|^{-\tau}),
\end{align*}
it then follows that
\begin{align}\label{Herzlich4}
\int_{\Omega_R}Q_{\hat{g}}\mathrm{div}_{\hat{g}}XdV_{\hat{g}}=n\int_{\Omega_R}Q_{\hat{g}}dV_{\hat{g}} + O(R^{n-2\tau-4}).
\end{align}

Let us notice that the AE condition together with $Q_g\in L^1(M,dV_g)$ imply that
\begin{align*}
\int_{\Omega_R}Q_{\hat{g}}dV_{\hat{g}}=\int_{\Omega_R}Q_{\hat{g}}dV_{\delta} + o(1).
\end{align*}
Using the above expression in the identity (\ref{Michell4}), we can rewrite
\begin{align*}
\int_{\Omega_R}Q_{\hat{g}}dV_{\hat{g}}&= \int_{S_R}\langle \mathbb{U}(h,1),\nu_{\delta}\rangle_{\delta}d\omega_{\delta} + \int_{S_{\frac{R}{4}}}\langle \mathbb{U}(h,1),\nu_{\delta}\rangle_{\delta}d\omega_{\delta} + \int_{\Omega_R}\mathcal{R}(h) dV_{\delta} +o(1),\\
&=\int_{S_R}\langle \mathbb{U}(h,1),\nu_{\delta}\rangle_{\delta}d\omega_{\delta}  + \int_{\Omega_R}\mathcal{R}(h) dV_{\delta} +o(1),
\end{align*}
where we have used (\ref{Michell5}) and the flatness in a neighbourhood of the inner boundary to eliminate the corresponding term. Furthermore, one deduces that under our conditions, since the left-hand side in (\ref{Michell4}) converges to $\mathcal{E}(g)$ due to Lemma \ref{MichellLemma2} and $Q_{g}\in L^1(M,dV_g)$, then the limit
\begin{align*}
\lim_{r\rightarrow\infty}\int_{B_r(0)}\mathcal{R}(h) dV_{\delta}<\infty,
\end{align*}
implying that, as $R\rightarrow\infty$,
\begin{align*}
\Big\vert\int_{\Omega_R}\mathcal{R}(h) dV_{\delta}\Big\vert=o(1).
\end{align*}
Therefore, one finds that
\begin{align*}
\int_{\Omega_R}Q_{\hat{g}}dV_{\hat{g}}&=\int_{S_R}\langle \mathbb{U}(h,1),\nu_{\delta}\rangle_{\delta}d\omega_{\delta}  + o(1).
\end{align*}
Using this in (\ref{Herzlich4}), we find that
\begin{align}\label{Herzlich5}
\int_{\Omega_R}Q_{\hat{g}}\mathrm{div}_{\hat{g}}XdV_{\hat{g}}=n\int_{S_R}\langle \mathbb{U}(h,1),\nu_{\delta}\rangle_{\delta}d\omega_{\delta} + o(1).
\end{align}
Thus, putting together (\ref{Herzlich2}),(\ref{Herzlich3}) and (\ref{Herzlich5}), we can rewrite (\ref{Herzlich1}) as
\begin{align*}
\int_{S_R}G_{J_{g}}(X,\nu_{g})d\omega_{g}&=  \frac{4-n}{4}\int_{S_R}\langle \mathbb{U}(h,1),\nu_{\delta}\rangle_{\delta}d\omega_{\delta} + o(1).
\end{align*}
Next, from Lemma \ref{MichellLemma2} we know that we can pass to the limit in the right-hand side of the above expression to get 
\begin{align}\label{linkEGJ}
\lim_{R\rightarrow\infty}\int_{S_R}G_{J_{g}}(X,\nu_{g})d\omega_{g}
&=-\frac{n-4}{8(n-1)}\mathcal{E}(g).
\end{align}

Finally, from our asymptotic conditions, we know that $G_{J_{g}}(X,\nu_{g})=G_{J_{g}}(X,\nu_{\delta})+O(|x|^{-2\tau-3})$ and $d\omega_g=(1+O(|x|^{-\tau}))d\omega_{\delta}$. Thus
\begin{align}\label{Herzlich6}
\int_{S_R}G_{J_{g}}(X,\nu_{g})d\omega_{g}
&=\int_{S_R}G_{J_{g}}(X,\nu_{\delta})d\omega_{\delta} + O(R^{n-2\tau-4}).
\end{align}
Using (\ref{linkEGJ}) and $n-2\tau-4<0$, we can pass to the limit and establish our result. 
\end{proof}

\begin{cor}
\label{linkJE}
Under the conditions of Theorem \ref{HerzlichThm}, if $J_g=0$, then $\mathcal{E}=0$.
\end{cor}
\begin{proof}
Since $\mathrm{tr}_g(J_g) = Q_g$, $J_g = 0$ implies $Q_g= 0$ and thus that $G_{J_g}=0$, which yields $\mathcal{E}(g)=0$ with \eqref{linkEGJ}.
\end{proof}

\section{$J_g$ flat means flat}

This section will be devoted to establishing a strong rigidity property of $J$-flat AE manifolds. That is, we shall prove that this property uniquely characterises Euclidean space, resembling the analogous property for Ricci-flat AE manifolds.


\begin{theo}\label{theoJflatisflat}
Let $(M^n,g)$ be a smooth AE manifold of class $W^{3,p}_{-\tau}(M,\Phi)$ with  $p>n \ge 3$ and $\tau>0$, for a given structure of infinity $\Phi$. If $J_g=0$ and $Y([g]) >0$, then $(M^n,g)$ is isometric to the Euclidean space $(\R^n, \delta)$.
\end{theo}

\begin{proof}
The proof will rely on a previous positive energy theorem. The corollary we present is a direct consequence of \cite[Theorem A]{avalos2021positive}:
\begin{cor}
\label{cortheoARJP}
Let $(M,g)$  be a smooth AE manifold  of dimension $n\ge 3$, of class $W^{4,\infty}_{-\tau}(M,\Psi)$ with $\tau > \max \left\{ 0, \frac{n-4}{2} \right\}$ in some coordinate system associated to a structure of infinity $\Psi$, and such that $Y([g])>0$ and $Q_g \ge 0$. Then, if  $\mathcal{E}(g)=0$, $(M,g)\cong (\mathbb{R}^n,\delta)$.
\end{cor}



{Appealing to Corollary \ref{cortheoARJP}, if $g$ is $W^{4,\infty}_{-\tau}(M,\Phi)$-AE with $\tau > \max \left\{ 0, \frac{n-4}{2} \right\}$, then the result stands, since from Corollary \ref{linkJE} we know that $J_g=0$ implies $\mathcal{E}(g)=0$. Then, we can apply corollary \ref{cortheoARJP} and conclude.}  {Therefore, the rest of the proof shall be devoted to showing that if $g$ is a priori $W^{3,p}_{-\tau}$ for some $\tau \le \max \left\{ 0, \frac{n-4}{2} \right\}$ in some end coordinate system $\Phi$ and $J_g=0$, then $g$ is $ W^{k,p}_{-\tau'}(M,\Theta)$-AE for all $\tau \le \tau' < n-4 $ and all integers $k$ for a harmonic structure of infinity $\Theta$. In such a case, once again, Corollary \ref{cortheoARJP} concludes the proof. Theorem \ref{J-Bootstrap} below establishes this bootstrap claim and hence finishes the proof.} 


\end{proof}


\begin{theo}\label{J-Bootstrap}
Let $(M^n,g)$ be a smooth $W^{3,p}_{-\tau}$-AE manifold, $p>n$, {$n\geq 3$ and $\tau>0$ with respect to some structure of infinity $\Phi$. If $J_g\in W^{k-4,p}_{-\delta-4}(M,\Phi)$, $k\geq 4$ and $\tau\leq \delta$, then there is a structure of infinity $\Theta$, given by harmonic end coordinates, such that $g$ is $W^{k,p}_{-\tau}(M,\Theta)$-AE. Furthermore, if $0<\tau<\delta<n-4$, $n\geq 5$,} then $g$ is $W^{k,p}_{-\delta}(M,\Theta)$-AE. 
\end{theo}

\begin{remark}\label{DecayBootsrapRemark}
{In practical terms, the above theorem gives a way to check whether one can bootstrap a $W^{3,p}_{-\tau}$-AE metric $g$: since by Theorem \ref{HarmonicCoordThm} changing to harmonic coordinates preserves the weighted Sobolev control on tensor fields, one can first fix harmonic coordinates, where we are guaranteed to have $g$ as $W^{3,p}_{-\tau}(M,\Theta)$-AE, and then check in these coordinates the behaviour of $J_g$, knowing a priori that any Sobolev control for it cannot be lost under this change of end coordinates. If $J_g$ remains controlled in terms of weighted Sobolev norms $W^{k,p}_{-\delta}(M,\Theta)$, Theorem \ref{J-Bootstrap} will bootstrap the decay for intermediary derivatives, and, when possible, also the order of their decay. This procedure is optimal, avoiding the search of other possible structure of infinity maintaining the original control on $g$ as well as an additional potential better control on $J_g$.}
\end{remark}

\begin{proof}
Let us start by proving the following claim, which shall give us an inductive argument:

\begin{claime}\label{BootstrapClaim}
Under the hypotheses of the theorem, if $g$ is $W^{l-1,p}_{-\tau}(M,\Phi)$-AE and $J_g\in W^{l-4,p}_{-\delta-4}(M,\Phi)$ with $l\geq 4$, then $g$ is $W^{l,p}_{-\tau}(M,\Theta)$-AE. Furthermore, if $0<\tau<\delta<n-4$, $n\geq 5$, then $g$ is $W^{l,p}_{-\sigma}(M,\Theta)$-AE for all $\tau \leq \sigma \leq \min\{2\tau,\delta\}$. 
\end{claime}

\begin{proof}
Since $(M,g)$ is an AE manifold of dimension $n\ge 3$ of class $W^{l-1,p}_{-\tau}$, $l\geq 4$, $0< \tau $, $p>n$, then, in end coordinates associated to $\Phi$, 
\begin{equation}\label{Ricreg1}
\mathrm{Ric}_{uv} \in W^{l-3,p}_{- \tau - 2}(\mathbb{R}^n\backslash\overline{B_1(0)},\Phi).
\end{equation}
Tracing the above also yields
\begin{equation}
\label{Rreg1}
{R}_g \in W^{l-3,p}_{-\tau-2}(M,\Phi),
\end{equation}
and thus, by  definition of the Schouten tensor:
\begin{equation}
\label{Sreg1}
S_g \in W^{l-3,p}_{-\tau-2}(M,\Phi).
\end{equation}

Let us now  recall the definition of \eqref{Qgintro}:
$$
\label{Qcurvdef}
Q_g \doteq - \frac{1}{2(n-1)} \Delta_g {R}_g - \frac{2}{(n-2)^2} \left| \mathrm{Ric}_g\right|^2_g + \frac{ n^3 - 4 n^2 + 16 n - 16}{8(n-1)^2 (n-2)^2 } {R}_g^2.
$$
Using once more Proposition \ref{Wkpdeltaalgebra} for $p >n$ and $l-3\geq 1$, given \eqref{Ricreg1} and \eqref{Rreg1}, we deduce $\left| \mathrm{Ric}_g\right|^2_g, \, {R}_g^2 \in W^{l-3,p}_{-2\tau-4}$. 
Since $\mathrm{tr}_g( J_g) = Q_g\in W^{l-4,p}_{-\delta-4}$, then \eqref{Qgintro} yields 
\begin{equation}\label{estdeltaR}
\Delta_g R_g \in W^{l-4,p}_{-\sigma_0 -2}(M,\Phi),
\end{equation}
for $\tau +2 \le \sigma_0 \doteq \min\{2\tau +2,\delta+2\}$. Notice then that $R_g\in W^{l-3,p}_{-\tau-2}(M,\Phi)$ and satisfies (\ref{estdeltaR}). Thus, we can apply Lemma \ref{lemregularite} to guarantee $R_g\in W^{l-2,p}_{-\tau-2}(M,\Phi)$. But now, we can  apply proposition \ref{regularitymaxwell} with \eqref{estdeltaR} and conclude that\footnote{{Notice that applying Lemma \ref{lemregularite} first was necessary to guarantee $R_g\in W^{l-3,p}_{-\tau-2}$, with $p>n$, and $l-2\geq 2$, which puts us under the hypotheses of Proposition \ref{regularitymaxwell}.}}${}^{,}$\footnote{{Notice that by hypothesis $\delta<n-4$ and hence $\delta+2<n-2$. Therefore, $\sigma_0\in (0,n-2)$, which is the range for the isomorphism claim in Proposition \ref{regularitymaxwell}.}} 

\begin{equation}\label{estR}
R_g \in W^{l-2,p}_{-\sigma}(M,\Phi)\text{ with } \begin{cases} \tau +2 \le \sigma \leq \min \left\{ 2\tau + 2 , \delta+2 \right\}, \text{ if } \tau<\delta<n-4\\
\sigma=\tau+2, \text{ otherwise }
\end{cases}
\end{equation}
Since $\mathrm{tr}_g (S_g) = \frac{{R}_g}{2(n-1)} $, we can rephrase and control the highest order component  of the $T_g$-tensor as defined by \eqref{T-tensor}:
\begin{align}\label{controleD2T}
\begin{split}
\nabla^2 \mathrm{tr}_g S_g - \frac{1}{n} \Delta_g \mathrm{tr}_g S_g g &= \frac{\nabla^2 {R}_g - \frac{1}{n} \Delta_g {R}_g g}{2(n-1)}  \in W^{l-4,p}_{- \sigma -2 } (M,\Phi)\\
  &\text{ with } \begin{cases} \tau +2 \le \sigma \leq \min \left\{ 2\tau + 2 , \delta+2 \right\}, \text{ if } \tau<\delta<n-4\\
\sigma=\tau+2, \text{ otherwise }
\end{cases}
\end{split}
\end{align}

Indeed,  $\nabla_{uv} {R}_g = \partial_{uv} {R}_g - \Gamma_{uv}^k \partial_k{R}_g$, with $\partial_{uv} {R}_g \in W^{l-4,p}_{- \sigma -2 } (\mathbb{R}^n\backslash\overline{B_1(0)})$, $\partial_k {R}_g \in W^{l-3,p}_{-\sigma -1}(\mathbb{R}^n\backslash\overline{B_1(0)})$. On the other hand, $\Gamma^k_{pq} \in W^{l-2,p}_{-\tau-1}(\mathbb{R}^n\backslash\overline{B_1(0)})$, and thus $\Gamma^k_{pq} \partial_k {R}_g \in W^{l-3,p}_{-\sigma - \tau -2}(\mathbb{R}^n\backslash\overline{B_1(0)}) \subset W^{l-3,p}_{-\sigma-2}(\mathbb{R}^n\backslash\overline{B_1(0)})$. The same reasoning yields $\Delta_g {R}_g \in W^{l-4,p}_{-\sigma-2}(M,\Phi)$, and thus \eqref{controleD2T}. 

Since all lower order terms in \eqref{T-tensor} are quadratic, they can be estimated as in  \eqref{estdeltaR} which yields 
\begin{equation}\label{controleT}
T_g \in W^{l-4,p}_{- \sigma -2 } (M,\Phi) \text{ with } \begin{cases} \tau +2 \le \sigma \leq \min \left\{ 2\tau + 2 , \delta+2 \right\}, \text{ if } \delta<n-4\\
\sigma=\tau+2, \text{ otherwise }
\end{cases}
\end{equation}

Going back to the expression of the $J_g$-tensor \eqref{J-tensor.2}, it can be pointed out that the tracefree part of $J_g$ is  a linear combination of the Bach tensor $B_g$ and $T_g$.  Under our hypothesis $J_g\in W^{l-4,p}_{-\delta-4}(M,\Phi)$ and appealing to \eqref{controleT}, this implies that 
\begin{equation}\label{controleB}
B_g \in W^{l-4,p}_{- \sigma -2 } (M,\Phi) \text{ with } \begin{cases} \tau +2 \le \sigma \leq \min \left\{ 2\tau + 2 , \delta+2 \right\}, \text{ if } \tau<\delta<n-4\\
\sigma=\tau+2, \text{ otherwise }
\end{cases}
\end{equation}

We  will use the following expression for the Bach tensor  (see \cite[Equation (2-6)]{aa43}):
\begin{equation}\label{BachtensorwithDelta}
\begin{aligned}
B_{uv}&=\frac{1}{n-2}\Delta_g \mathrm{Ric}_{uv} - \frac{1}{2(n-1)(n-2)} \Delta_g {R}_g g_{uv} - \frac{1}{2(n-1)} \nabla_{uv} {R}_g  + 2 \mathrm{Riem}_{auvb} S^{ab}\\ & -(n-4) S^a_u S_{av}  - \left|{S}_g\right|^2g_{uv} - 2 \mathrm{Tr}(S_g) S_{uv},
\end{aligned}
\end{equation}
where 
$\Delta_g \mathrm{Ric}_{uv}$ denotes the tensorial Laplacian expressed in the chosen end-chart:
$$\begin{aligned}
\Delta_g \mathrm{Ric}_{uv} &= g^{ab} \left[ \partial_a \left( \partial_b \mathrm{Ric}_{uv}- \Gamma^k_{bu} \mathrm{Ric}_{kv}  - \Gamma^k_{bv} \mathrm{Ric}_{uk} \right) - \Gamma^k_{ab} \left( \partial_k \mathrm{Ric}_{uv} - \Gamma^l_{ku} \mathrm{Ric}_{lv} - \Gamma^l_{vk} \mathrm{Ric}_{ul} \right) \right. \\& \left.- \Gamma^k_{au} \nabla_b \mathrm{Ric}_{kv}- \Gamma^k_{av} \nabla_b \mathrm{Ric}_{uk}\right] \\
&= g^{ab} \left[ \partial_{ab} \mathrm{Ric}_{uv} - \Gamma^k_{ab} \partial_k \mathrm{Ric}_{uv}  \right.  \\& - \left. \left\{ \partial_a \left( \Gamma^k_{bu} \mathrm{Ric}_{kv}+ \Gamma^k_{bv} \mathrm{Ric}_{uk} \right) - \Gamma^k_{ab}  \Gamma^l_{ku} \mathrm{Ric}_{lv}-  \Gamma^k_{ab} \Gamma^l_{vk} \mathrm{Ric}_{ul}  +  \Gamma^k_{au}\nabla_b \mathrm{Ric}_{kv} + \Gamma^k_{av} \nabla_b\mathrm{Ric}_{uk} \right\} \right].
\end{aligned}$$

For analytic convenience, we will write 
\begin{equation}
\label{LaplBeltramiRicci}
\Delta_g \mathrm{Ric}_{uv} = \Delta_g  \left(\mathrm{Ric}_{uv} \right) + \mathrm{E}_{uv},
\end{equation}
with 
\begin{equation}
\label{ResteLaplace}
\mathrm{E}_{uv} =g^{ab}\left[ \partial_a \left( \Gamma^k_{bu} \mathrm{Ric}_{kv}+ \Gamma^k_{bv} \mathrm{Ric}_{uk} \right) - \Gamma^k_{ab}  \Gamma^l_{ku} \mathrm{Ric}_{lv}-  \Gamma^k_{ab} \Gamma^l_{vk} \mathrm{Ric}_{ul}  +  \Gamma^k_{au}\nabla_b \mathrm{Ric}_{kv} + \Gamma^k_{av} \nabla_b\mathrm{Ric}_{uk} \right],
\end{equation}
and where $\Delta_g \left( \mathrm{Ric}_{uv} \right) $ denotes the Laplace-Beltrami operator for \emph{functions} applied to the function defined by the $u,v$ component  of the Ricci tensor in our choice of end coordinates.

As in the proof of  \eqref{estdeltaR}, thanks  to Proposition \ref{Wkpdeltaalgebra},  all the quadratic terms in \eqref{BachtensorwithDelta} are  in $W^{l-3,p}_{-\sigma - 2 } (M,\Phi)$. 
In addition, as in \eqref{controleD2T}, since \eqref{estR} stands, $\Delta_g {R}_g g_{uv} , \nabla_{uv}R_g \in W^{l-4,p}_{-\sigma -2}(M,\Phi)$. Injecting all this in \eqref{controleB} then yields:
\begin{equation}\label{estDeltaRic}
\Delta_g \mathrm{Ric}_{uv} \in  W^{l-4,p}_{-\sigma -2}(\mathbb{R}^n\backslash\overline{B_1(0)},\Phi) \quad\text{ with } \begin{cases} \tau +2 \le \sigma \leq \min \left\{ 2\tau + 2 , \delta+2 \right\}, \text{ if } \tau<\delta<n-4\\
\sigma=\tau+2, \text{ otherwise }
\end{cases} 
\end{equation}

Injecting \eqref{estDeltaRic} into \eqref{LaplBeltramiRicci} then ensures that for all $ u,v$:
\begin{equation}\label{estDeltaRicLB1}
\quad \Delta_g \left( \mathrm{Ric}_{uv} \right) = - \mathrm{E}_{uv} + W^{l-4,p}_{-\sigma -2}(\mathbb{R}^n\backslash\overline{B_1(0)},\Phi) \text{ with } \begin{cases} \tau +2 \le \sigma \leq \min \left\{ 2\tau + 2 , \delta+2 \right\}, \text{ if } \tau<\delta<n-4\\
\sigma=\tau+2, \text{ otherwise }
\end{cases}
\end{equation}
We then need to control the remainder term $\mathrm{E}_{uv}$. Since the Christoffel symbols are known to be in $W^{l-2,p}_{-\tau-1}$, and with \eqref{Ricreg1} $\mathrm{Ric} \in W^{l-3,p}_{-\tau-2}(M,\Phi)$, we can deduce that $\partial \mathrm{Ric} \in W^{l-4,p}_{-\tau-3}$ and $\Gamma \mathrm{Ric} \in W^{l-3,p}_{-2\tau -3}$, which yields $\nabla \mathrm{Ric} \in W^{l-4,p}_{-\tau-3}(M,\Phi)$.
Hence:
\begin{equation} \label{Tdecompoen3}\begin{aligned}
 &\Gamma^k_{au}\nabla_b \mathrm{Ric}_{kv} + \Gamma^k_{av}\nabla_b \mathrm{Ric}_{uk} \in W^{l-4,p}_{-2\tau -4}(M,\Phi) \\
& \Gamma^k_{ab}  \Gamma^l_{ku} \mathrm{Ric}_{lv}+  \Gamma^k_{ab} \Gamma^l_{uk} \mathrm{Ric}_{vl} \in W^{l-3,p}_{-3\tau - 4}(M,\Phi) \subset W^{l-4,p}_{-2\tau -4}(M,\Phi) \\ &\partial_a \left( \Gamma^k_{bu} \mathrm{Ric}_{kv}+ \Gamma^k_{bv} \mathrm{Ric}_{uk} \right) \in W^{l-4,p}_{-2\tau -4}(M,\Phi).  
\end{aligned}
\end{equation}
Together \eqref{ResteLaplace} and \eqref{Tdecompoen3} ensure that $\mathrm{E}_{uv} \in W^{l-4,p}_{-2\tau -4}(M,\Phi)$, and injected into \eqref{estDeltaRicLB1} yields for all $u,v$:
\begin{equation}\label{estDeltaRicLB2}
\Delta_g \left( \mathrm{Ric}_{uv} \right)  \in W^{l-4,p}_{-\sigma -2}(\mathbb{R}^n\backslash\overline{B_1(0)},\Phi)\text{ with } \begin{cases} \tau +2 \le \sigma \leq \min \left\{ 2\tau + 2 , \delta+2 \right\}, \text{ if } \tau<\delta<n-4\\
\sigma=\tau+2, \text{ otherwise }
\end{cases}
\end{equation}

Thus, since a priori $\mathrm{Ric}_g\in W^{l-3,p}_{-\tau-2}(M,\Phi)$, $l-3\geq 1$, we can first apply Lemma \ref{lemregularite} to bootstrap from (\ref{estDeltaRicLB2}) to $\mathrm{Ric}_g\in W^{l-2,p}_{-\tau-2}(M,\Phi)$, with $l-2\geq 2$, and then thanks to Proposition \ref{regularitymaxwell}, \eqref{Ricreg1} and \eqref{estDeltaRicLB2} yield
\begin{equation}\label{estRic}
\mathrm{Ric}_{uv} \in  W^{l-2,p}_{-\sigma}(\mathbb{R}^n\backslash\overline{B_1(0)},\Phi) \text{ with } \begin{cases} \tau +2 \le \sigma \leq \min \left\{ 2\tau + 2 , \delta+2 \right\}, \text{ if } \tau<\delta<n-4\\
\sigma=\tau+2, \text{ otherwise }
\end{cases}.
\end{equation}

Now, appealing to Theorem \ref{HarmonicCoordThm}, we know that $\mathrm{Ric}_g\in W^{l-2,p}_{-\sigma}(M,\Theta)$ where $\Theta$ denotes harmonic structure of infinity. Also, in the corresponding harmonic end coordinates, Theorem \ref{HarmonicCoordThm} also guarantees that $g_{ij}-\delta_{ij}\in W^{l-1,p}_{-\tau}(\Theta)$. Thus, appealing to \eqref{RicciHarmonic}, we can write the following expression in harmonic coordinates:
\begin{equation} \label{rewritericci}\begin{aligned}
\mathrm{Ric}_{uv} &= -\frac{1}{2}g^{ab} \partial_{ab}g_{uv}  +D_{uv}(g, \partial g)  \\
&= -\frac{1}{2} \Delta_\delta \left[g_{uv} - \delta_{uv} \right]  -\frac{1}{2} \left( g^{ab} - \delta^{ab} \right) \partial_{ab} g_{uv} + D_{uv} (g, \partial g),
\end{aligned}\end{equation}
where $\Delta_\delta$ denotes the flat Laplacian on the structure of infinity $\Theta$, on which we now work.

We know that $D_{uv}(g, \partial g) \in W^{l-2,p}_{-2\tau-2 }(\mathbb{R}^n\backslash\overline{B_1(0)},\Theta)$. Since $ g^{ab} - \delta^{ab} \in W^{l-1,p}_{-\tau} (\mathbb{R}^n\backslash\overline{B_1(0)},\Theta)$ and  $\partial_{ab} g_{uv} \in  W^{l-3,p}_{-\tau-2} (\mathbb{R}^n\backslash\overline{B_1(0)},\Theta)$, we also have that
\begin{align*}
\left( g^{ab} - \delta^{ab} \right) \partial_{ab} g_{uv} \in W^{l-3,p}_{-2\tau -2}(\mathbb{R}^n\backslash\overline{B_1(0)},\Theta) \subset  W^{l-3,p}_{-\sigma}(\mathbb{R}^n\backslash\overline{B_1(0)},\Theta).
\end{align*}
Thanks to \eqref{estRic}, we can then rewrite \eqref{rewritericci} into
\begin{align*}
\Delta_\delta \left[ g_{uv}- \delta_{uv}\right] \in W^{l-2,p}_{-\sigma}(\mathbb{R}^n\backslash\overline{B_1(0)},\Theta) \text{ with } \begin{cases} \tau +2 \le \sigma \leq \min \left\{ 2\tau + 2 , \delta+2 \right\}, \text{ if } \tau<\delta<n-4\\
\sigma=\tau+2, \text{ otherwise }
\end{cases}
\end{align*}
Since $g_{uv}- \delta_{uv} \in W^{l-1,p}_{-\tau}(\mathbb{R}^n\backslash\overline{B_1(0)},\Theta)$, $l-1\geq 3$, Proposition \ref{regularitymaxwell} yields that 
\begin{equation} 
g_{uv}-\delta_{uv} \in W^{l,p}_{-\left[ \sigma -2 \right]}(\mathbb{R}^n\backslash\overline{B_1(0)},\Theta) \text{ with } \begin{cases} \tau +2 \le \sigma \leq \min \left\{ 2\tau + 2 , \delta+2 \right\}, \text{ if } \tau<\delta<n-4\\
\sigma=\tau+2, \text{ otherwise }
\end{cases}
\end{equation}
\end{proof}

Using the above claim, if $g$ is $W^{3,p}_{-\tau}(M,\Phi)$-AE and $J_g\in W^{k-4,p}_{-\delta-4}(M,\Phi)$, $k\geq 4$, then we can start an iterative bootstrap, which shall increase both the number of derivatives decaying as well as the rate of decay. To simplify the argument, let us notice that Claim \ref{BootstrapClaim} gives the possibility to decouple these processes. Thus, concerning the number of decaying of derivatives, in the $j$-th step we bootstrap $g$ to $W^{3+j,p}_{-\tau}(M,\Theta)$-AE, as long as $3+j\leq k-1$, and hence the bootstrap on the number of derivative only ends when we reach $g$ as $W^{k,p}_{-\tau}(M,\Theta)$-AE. Then, if $\tau<\delta<n-4$, we can bootstrap the order of decay as follows using Claim \ref{BootstrapClaim}.

First, notice that Theorem \ref{HarmonicCoordThm} implies that $J_g\in W^{k-4,p}_{-\delta-4}(M,\Theta)$. Thus, since $g$ is $W^{k,p}_{-\tau}(M,\Theta)$-AE and $J_g\in W^{k-4,p}_{-\delta-4}(M,\Theta)$, then Claim \ref{BootstrapClaim} gives $g$ as $W^{k,p}_{-\sigma}(M,\Theta)$-AE, for all $\sigma\leq \min \left\{ 2\tau , \delta \right\}$. Set $\tau_0\doteq \tau$, and build a sequence $\{\tau_i\}_{i\geq 0}$ as follows:
\begin{enumerate}
\item Starting with $\sigma=\tau_0$ and $g$ as $W^{k,p}_{-\sigma}(M,\Theta)$-AE, if $\min \left\{ 2\tau_i , \delta \right\}=\delta$, then we obtain $g$ as $W^{k,p}_{-\delta}(M,\Theta)$-AE and the procedure stops;
\item If $2\tau_i< \delta$, set $\tau_{i+1}\doteq 2\tau_i$. Since $g$ is $W^{k,p}_{-2\tau_i}(M,\Theta)$-AE and $J_g\in W^{k-4,p}_{-\delta-4}(M,\Theta)$ with $2\tau_i<\delta$, then Claim \ref{BootstrapClaim} gives $g$ is $W^{k,p}_{-\min\{2\tau_i,\delta\}}(M,\Theta)$-AE; 
\item Now, go back to item 1 starting with $\sigma=\min\{2\tau_i,\delta\}$ and iterate.
\end{enumerate}
Since as long as $2\tau_i<\delta$, we have $\tau_{i+1}=2\tau_i=2^{i}\tau$, after a finite number of loops we must find $2\tau_{i+1}>\delta$ and, in that iteration, then the bootstrap stops at the first step in this algorithm.

\end{proof}

\addcontentsline{toc}{section}{References}
\printbibliography

@article{Graf,
author = {C. Cederbaum and M. Graf and J. Metzger},
title = {{Initial data sets that do not satisfy
the Regge-Teitelboim conditions}},
volume = {},
journal = {},
number = {},
publisher = {},
pages = {},
year = {Work in progress},
doi = {},
URL = {}
}

@article{CB-C,
author = {Y. Choquet-{B}ruhat and D. Christodoulou},
title = {{Elliptic systems in $H_{s,\delta}$ spaces on manifolds which are euclidean at infinity}},
volume = {146},
journal = {Acta Mathematica},
number = {none},
publisher = {Institut Mittag-Leffler},
pages = {129 -- 150},
year = {1981},
doi = {10.1007/BF02392460},
URL = {https://doi.org/10.1007/BF02392460}
}

@article{DouglisNirenberg,
title = {The null spaces of elliptic partial differential operators in Rn},
journal = {Journal of Mathematical Analysis and Applications},
volume = {42},
number = {2},
pages = {271-301},
year = {1973},
issn = {0022-247X},
doi = {https://doi.org/10.1016/0022-247X(73)90138-8},
url = {https://www.sciencedirect.com/science/article/pii/0022247X73901388},
author = {L. Nirenberg and H. Walker}
}

@article{avalos2021positive,
author = {Avalos, R. and Laurain, P. and Lira, J. H.},
title = {A positive energy theorem for fourth-order gravity},
journal = {Calc. Var.},
volume = {61},
number = {48},
pages = { },
year = {2022},
doi = {10.1007/s00526-021-02152-w},
}

@article{Maxwell-Dilts,
author = {J. Dilts and D.Maxwell},
title = {Yamabe classification and prescribed scalar curvature in the asymptotically Euclidean setting},
journal = {Commun. Anal. Geom.},
volume = {26},
number = {5},
pages = { 1127–1168},
year = {2018},
doi = {10.4310/CAG.2018.v26.n5.a5},
}

@Book{Adams,
 Author = {Adams, R. and Fournier, J.},
 Title = {Sobolev {S}paces},
 Edition = {2nd ed.},
 ISBN = {0-12-044143-8},
 Year = {2003},
 Publisher = {New York, NY: Academic Press},
 Language = {English},
 zbMATH = {2208228},
 Zbl = {1098.46001}
}

@Book{Leoni,
 Author = {Leoni, G.},
 Title = {A First Course in Sobolev spaces},
 ISBN = {9780821847688},
 Year = {2009},
 Publisher = {American Mathematical Society},
 Language = {English}
}

@Book{GirbauBruna,
 Author = {Girbau, J. and Bruna, L.},
 Title = {Stability by linearization of {Einstein}'s field equation},
 FSeries = {Progress in Mathematical Physics},
 Series = {Prog. Math. Phys.},
 ISSN = {1544-9998},
 Volume = {58},
 ISBN = {978-3-0346-0303-4; 978-3-0346-0304-1},
 Year = {2010},
 Publisher = {Basel: Birkh{\"a}user},
 Language = {English},
 DOI = {10.1007/978-3-0346-0304-1},
 zbMATH = {5633771},
 Zbl = {1202.35003}
}

@book{AMR-Book,
  title={Manifolds, Tensor Analysis, and Applications},
  author={Abraham, R. and Marsden, J.E. and Ratiu, T.},
  isbn={9781461210290},
  lccn={89093238},
  series={Applied Mathematical Sciences},
  url={},
  year={2012},
  publisher={Springer New York}
}

@ARTICLE{Cantor-SplittingTensors,
    author = {Cantor, M.},
    title = {Elliptic {O}perators and the {D}ecomposition of {T}ensor {F}ields},
    volume = {5},
    journal = {Bull. Amer. Math. Soc.},
    fjournal = {Bulletin of the American Mathematical Society},
    year = {1981},
    pages = {235-262},
    issn = {},
    mrclass = {},
    mrnumber = {},
    mrreviewer = {},
    doi = {10.1090/S0273-0979-1981-14934-X},
    url = {},
    zbl = {0481.58023},
}

@article{Malchiodi1,
	author = {Gursky, M. and Malchiodi, A.},
	doi = {10.4171/JEMS/553},
	fjournal = {Journal of the European Mathematical Society (JEMS)},
	issn = {1435-9855},
	journal = {J. Eur. Math. Soc. (JEMS)},
	mrclass = {53C44 (35B50 35K58 35R01 53A30)},
	mrnumber = {3420504},
	mrreviewer = {Gang Li},
	number = {9},
	pages = {2137--2173},
	title = {A strong maximum principle for the {P}aneitz operator and a non-local flow for the {$Q$}-curvature},
	url = {https://doi.org/10.4171/JEMS/553},
	volume = {17},
	year = {2015}
}

@article{Raulot,
	author = {Humbert, E. and Raulot, S.},
	doi = {10.1007/s00526-009-0241-6},
	fjournal = {Calculus of Variations and Partial Differential Equations},
	issn = {0944-2669},
	journal = {Calc. Var. Partial Differential Equations},
	mrclass = {53C21 (35J30 35R01)},
	mrnumber = {2558328},
	mrreviewer = {Gabjin Yun},
	number = {4},
	pages = {525--531},
	title = {Positive mass theorem for the {P}aneitz-{B}ranson operator},
	url = {https://doi.org/10.1007/s00526-009-0241-6},
	volume = {36},
	year = {2009}
}

@article{hangyang2,
	author = {Hang, F. and Yang, P.},
	doi = {10.1093/imrn/rnu247},
	fjournal = {International Mathematics Research Notices. IMRN},
	issn = {1073-7928},
	journal = {Int. Math. Res. Not. IMRN},
	mrclass = {53A30 (35J08 53C20)},
	mrnumber = {3431611},
	mrreviewer = {Ali Maalaoui},
	number = {19},
	pages = {9775--9791},
	title = {Sign of {G}reen's function of {P}aneitz operators and the {$Q$} curvature},
	url = {https://doi.org/10.1093/imrn/rnu247},
	year = {2015}
}

@article{LinYuan2,
author = {YJ. Lin and W. Yuan,},
title = {Deformations of Q-curvature II},
journal = {Calc. Var.},
volume = {61},
number = {},
pages = {74},
year = {2022},
doi = {10.1007/s00526-021-02181-5},
}

@article{ChangGursk,
author = {S.-Y.A.Chang and M. Gursky and P. Yang},
title = {Remarks on a fourth order invariant in conformal geometry},
journal = {Aspects of Mathematics},
volume = {55},
number = {},
pages = {353 - 372},
year = {},
doi = {},
}

@article{LinYuan1,
author = {YJ. Lin and W. Yuan,},
title = {Deformations of Q-curvature I},
journal = {Calc. Var.},
volume = {55},
number = {},
pages = {101},
year = {2016},
doi = {10.1007/s00526-016-1038-z},
}

@article{Corvino1,
author = {J. Corvino},
title = {Scalar Curvature Deformation and a Gluing Construction for the Einstein Constraint Equations},
journal = {Commun. Math. Phys.},
volume = {214},
number = {},
pages = {137–189},
year = {2000},
doi = {10.1007/PL00005533},
}

@article{Miao2,
title = "Evaluation of the adm mass and center of mass via the ricci tensor",
author = "P. Miao and Tam, {L.}",
note = "Publisher Copyright: {\textcopyright} 2015 American Mathematical Society.",
year = "2016",
month = feb,
doi = "10.1090/proc12726",
language = "English (US)",
volume = "144",
pages = "753--761",
journal = "Proceedings of the American Mathematical Society",
issn = "0002-9939",
publisher = "American Mathematical Society",
number = "2",
}

@article{Lee-Parker,
	author = {Lee, J. and Parker, T.},
	doi = {10.1090/S0273-0979-1987-15514-5},
	fjournal = {American Mathematical Society. Bulletin. New Series},
	issn = {0273-0979},
	journal = {Bull. Amer. Math. Soc. (N.S.)},
	mrclass = {53-02 (35J60 53A30 53C20 58E15 58G30)},
	mrnumber = {888880},
	mrreviewer = {J. L. Kazdan},
	number = {1},
	pages = {37--91},
	title = {The {Y}amabe problem},
	url = {https://doi.org/10.1090/S0273-0979-1987-15514-5},
	volume = {17},
	year = {1987}
}

@article{Schoen1,
	author = {Schoen, R.},
	fjournal = {Journal of Differential Geometry},
	issn = {0944-2669},
	journal = {J. Diff. Geom},
	mrclass = {},
	mrnumber = {},
	number = {},
	pages = {479--595},
	title = {Conformal deformation of a Riemannian metric to constant scalar curvature},
	url = {https://doi.org/10.1007/s005260100134},
	volume = {20},
	year = {1984}
}

@article{SY1,
	author = {Schoen, R. and Yau, S.},
	fjournal = {Communications in Mathematical Physics},
	issn = {},
	journal = {Comm. Math. Phys},
	mrclass = {},
	mrnumber = {},
	mrreviewer = {},
	number = {},
	pages = {45--76},
	title = {On the proof of the positive mass conjecture in general relativity},
	url = {https://aip.scitation.org/doi/10.1063/5.0008749},
	volume = {65},
	year = {1979}
}

@article{SY2,
	author = {Schoen, R. and Yau, S.},
	fjournal = {Communications in Mathematical Physics},
	issn = {},
	journal = {Comm. Math. Phys},
	mrclass = {},
	mrnumber = {},
	mrreviewer = {},
	number = {},
	pages = {231--260},
	title = {},
	volume = {79},
	year = {1981}
}

@article{SY3,
	author = {Schoen, R. and Yau, S.},
	fjournal = {Inventiones Mathematicae},
	issn = {},
	journal = {Invent. math.},
	mrclass = {},
	mrnumber = {},
	mrreviewer = {},
	number = {},
	pages = {47--71},
	title = {Conformally flat manifolds, Kleinian groups and scalar curvature},
	volume = {92},
	year = {1988}
}

@article{Carlotto1,
	author = {Carlotto, A.},
	fjournal = {Calculus of Variations and Partial Differential Equations},
	issn = {},
	journal = {Calc. Var. Partial Differ. Equ.},
	mrclass = {},
	mrnumber = {},
	mrreviewer = {},
	number = {3},
	pages = {1--20},
	title = {Rigidity of stable minimal hypersurfaces in asymptotically flat spaces},
	volume = {55},
	year = {1988}
}

@article{Carlotto2,
	author = {Carlotto, A. and Chodosh, O. and Eichmair, M.},
	fjournal = {Inventiones mathematicae},
	issn = {},
	journal = {Invent. Math.},
	mrclass = {},
	mrnumber = {},
	mrreviewer = {},
	number = {3},
	pages = {1--20},
	title = {Effective versions of the {P}ositive {M}ass {T}heorem},
	volume = {206},
	year = {2016}
}

@article{Carlotto3,
	author = {Carlotto, A. and Schoen, R.},
	fjournal = {Inventiones mathematicae},
	issn = {},
	journal = {Invent. Math.},
	mrclass = {},
	mrnumber = {},
	mrreviewer = {},
	number = {3},
	pages = {1--20},
	title = {Localizing solutions of the {E}instein {C}onstraint {E}quations},
	volume = {205},
	year = {2016}
}

@book{Carlotto4,
	author = {Carlotto, .},
	isbn = {},
	mrclass = {},
	mrnumber = {},
	pages = {},
	publisher = {in: Einstein Equations: Physical and Mathematical Aspects of General Relativity. Eds: S. Cacciatori, Batu G{\"u}neysu and S. Pigola. Springer Nature Switzerland AG},
	series = {},
	title = {Four Lectures on Asymptotically Flat Riemannian Manifolds},
	volume = {},
	year = {2019}
}

@article{Carlotto4-2,
	author = {Carlotto, A.},
	fjournal = {},
	issn = {},
	journal = {Invent. Math.},
	mrclass = {},
	mrnumber = {},
	mrreviewer = {},
	number = {3},
	pages = {1--20},
	title = {Localizing solutions of the Einstein constraint equations},
	volume = {205},
	year = {2016}
}

@article{Eichmair1,
	author = {Brendle, S. and Eichmair, M.},
	fjournal = {Inventiones mathematicae},
	issn = {},
	journal = {Invent. Math.},
	mrclass = {},
	mrnumber = {},
	mrreviewer = {},
	number = {3},
	pages = {663--682},
	title = {Large {O}utlying  {S}table {C}onstant {M}ean {C}urvature {S}pheres in {I}nitial {Da}ta {S}ets},
	volume = {107},
	year = {2014}
}

@article{Yau1,
	author = {Huisken, G. and Yau, S.},
	fjournal = {Inventiones mathematicae},
	issn = {},
	journal = {Invent. Math.},
	mrclass = {},
	mrnumber = {},
	mrreviewer = {},
	number = {},
	pages = {281--311},
	title = {Definition of center of mass for isolated physical systems and unique foliations by stable spheres with constant mean curvature},
	volume = {124},
	year = {1996}
}

@article{Eichmair2,
	author = {Eichmair, M. and Metzger, J.},
	fjournal = {Journal of Differential Geometry},
	issn = {},
	journal = {J. Differ. Geom.},
	mrclass = {},
	mrnumber = {},
	mrreviewer = {},
	number = {1},
	pages = {159--186},
	title = {Large isoperimetric surfaces in initial data sets},
	volume = {94},
	year = {2013}
}

@article{AvalosFreitas,
title = {The Pohozaev-Schoen identity on asymptotically Euclidean manifolds: Conservation laws and their applications},
journal = {Annales de l'Institut Henri Poincaré C, Analyse non linéaire},
volume = {38},
number = {6},
pages = {1703-1724},
year = {2021},
issn = {0294-1449},
doi = {https://doi.org/10.1016/j.anihpc.2021.01.002},
url = {https://www.sciencedirect.com/science/article/pii/S0294144921000226},
author = {R. Avalos and A. Freitas},
}

@article{Herzlich,
author = {Herzlich, M. },
title = {Computing Asymptotic Invariants with the Ricci Tensor on Asymptotically Flat and Asymptotically Hyperbolic Manifolds},
journal = {Ann. Henri Poincaré },
volume = {17},
number = {},
pages = {3605–3617},
year = {2016},
doi = {},
}

@article{Michel,
author = {Michel,B. },
title = {Geometric invariance of mass-like asymptotic invariants},
journal = {Journal of Mathematical Physics},
volume = {52},
number = {5},
pages = {052504},
year = {2011},
doi = {10.1063/1.3579137},

URL = { 
        https://doi.org/10.1063/1.3579137
    
},
eprint = { 
        https://doi.org/10.1063/1.3579137
    
}
}

@article {MR2116728,
    AUTHOR = {Maxwell, D.},
     TITLE = {Solutions of the {E}instein constraint equations with apparent
              horizon boundaries},
   JOURNAL = {Comm. Math. Phys.},
  FJOURNAL = {Communications in Mathematical Physics},
    VOLUME = {253},
      YEAR = {2005},
    NUMBER = {3},
     PAGES = {561--583},
      ISSN = {0010-3616},
   MRCLASS = {83C05 (58J45 83C57)},
  MRNUMBER = {2116728},
MRREVIEWER = {Simonetta Frittelli},
       DOI = {10.1007/s00220-004-1237-x},
       URL = {https://doi.org/10.1007/s00220-004-1237-x},
}

@article {MR849427,
    AUTHOR = {Bartnik, R.},
     TITLE = {The mass of an asymptotically flat manifold},
   JOURNAL = {Comm. Pure Appl. Math.},
  FJOURNAL = {Communications on Pure and Applied Mathematics},
    VOLUME = {39},
      YEAR = {1986},
    NUMBER = {5},
     PAGES = {661--693},
      ISSN = {0010-3640},
   MRCLASS = {58G30 (46E35 53C80 58D30 83C30)},
  MRNUMBER = {849427},
MRREVIEWER = {M. A. H. MacCallum},
       DOI = {10.1002/cpa.3160390505},
       URL = {https://doi.org/10.1002/cpa.3160390505},
}

@article {MR456272,
    AUTHOR = {Coll, B.},
     TITLE = {On the evolution equations for {K}illing fields},
   JOURNAL = {J. Mathematical Phys.},
  FJOURNAL = {Journal of Mathematical Physics},
    VOLUME = {18},
      YEAR = {1977},
    NUMBER = {10},
     PAGES = {1918--1922},
      ISSN = {0022-2488},
   MRCLASS = {83.35},
  MRNUMBER = {456272},
       DOI = {10.1063/1.523164},
       URL = {https://doi.org/10.1063/1.523164},
}

@article {MR416469,
    AUTHOR = {Moncrief, V.},
     TITLE = {Space-time symmetries and linearization stability of the
              {E}instein equations. {II}},
   JOURNAL = {J. Mathematical Phys.},
  FJOURNAL = {Journal of Mathematical Physics},
    VOLUME = {17},
      YEAR = {1976},
    NUMBER = {10},
     PAGES = {1893--1902},
      ISSN = {0022-2488},
   MRCLASS = {83.53 (58D05)},
  MRNUMBER = {416469},
MRREVIEWER = {M. Carmeli},
       DOI = {10.1063/1.522814},
       URL = {https://doi.org/10.1063/1.522814},
}

@article{aa43,
title = "A Symmetric 2-Tensor canonically associated to Q-curvature and its applications",
author = "Lin, {Y.} and Wei Y.",
note = "Publisher Copyright: {\textcopyright} 2017 Mathematical Sciences Publishers.",
year = "2017",
doi = "10.2140/pjm.2017.291.425",
language = "English (US)",
volume = "291",
pages = "425--438",
journal = "Pacific Journal of Mathematics",
issn = "0030-8730",
publisher = "University of California, Berkeley",
number = "2",
}

@misc{avalos2021energy,
      title={Energy in Fourth Order Gravity}, 
      author={Avalos, R. and Lira, J. and Marque, N.},
       journal = {Annales Henri Poincaré},
    fjournal = {Annales Henri Poincaré},
    volume = {26},
    year = {2025},
    pages = {597--673},
    issn = {1424-0661},
    mrclass = {},
    mrnumber = {},
    mrreviewer = {},
    doi = {10.1007/s00023-024-01440-3},
    url = {https://doi.org/10.1007/s00023-024-01440-3},
    zbl = {},
}

@article {MR2225517,
    AUTHOR = {Corvino, J. and Schoen, R.},
     TITLE = {On the asymptotics for the vacuum {E}instein constraint
              equations},
   JOURNAL = {J. Differential Geom.},
  FJOURNAL = {Journal of Differential Geometry},
    VOLUME = {73},
      YEAR = {2006},
    NUMBER = {2},
     PAGES = {185--217},
      ISSN = {0022-040X},
   MRCLASS = {58J45 (35Q75 83C05)},
  MRNUMBER = {2225517},
MRREVIEWER = {Alan D. Rendall},
       URL = {http://projecteuclid.org/euclid.jdg/1146169910},
}

@book {MR2473363,
    AUTHOR = {Choquet-Bruhat, Y.},
     TITLE = {General relativity and the {E}instein equations},
    SERIES = {Oxford Mathematical Monographs},
 PUBLISHER = {Oxford University Press, Oxford},
      YEAR = {2009},
     PAGES = {xxvi+785},
      ISBN = {978-0-19-923072-3},
   MRCLASS = {83-02 (35Q75 83C05)},
  MRNUMBER = {2473363},
MRREVIEWER = {Hans-Peter K\"{u}nzle},
}

\end{document}